\documentclass[11pt, reqno]{amsart}

\usepackage{amsmath,amssymb,amsthm, epsfig}
\usepackage{hyperref}
\usepackage[font=small,textfont=it]{caption}
\hypersetup{colorlinks=true, linkcolor=blue, citecolor=red, linktoc=page, pdftitle={Sharp regularity for singular obstacle problems}, pdfauthor={D.J. Ara\'ujo, R. Teymurazyan, V. Voskanyan}}
\usepackage{ulem}
\usepackage[mathscr]{euscript}
\usepackage[latin1]{inputenc}
\usepackage{xcolor}
\usepackage{hyperref}
\usepackage[nameinlink]{cleveref}
\hypersetup{
	colorlinks,
	linkcolor={blue!50!black},
	citecolor={blue!50!black},
	urlcolor={blue!50!black}
}

\usepackage{enumerate}
\usepackage{mathrsfs}
\usepackage[usenames,dvipsnames]{pstricks}
\usepackage{pst-grad} 
\usepackage{pst-plot} 

\newtheorem{theorem}{Theorem}

\newtheorem{lemma}{Lemma}

\newtheorem{corollary}{Corollary}
\newtheorem{remark}{Remark}

\date{}
\numberwithin{equation}{section}
\numberwithin{theorem}{section}
\numberwithin{lemma}{section}
\numberwithin{corollary}{section}
\numberwithin{remark}{section} 
\numberwithin{proposition}{section}
\numberwithin{definition}{section}

\newcommand{\Div}{\operatorname{div}}

\newcommand{\R}{\mathbb{R}}

\newcommand{\dist}{\operatorname{dist}}
\newcommand{\loc}{\operatorname{loc}}

\begin{document}
	
	\title[Sharp regularity for singular obstacle problems]{Sharp regularity for singular obstacle problems}
	
	\author[D.J. Ara\'ujo]{Dami\~ao J. Ara\'ujo}
	\address{UFPB, Department of Mathematics, Universidade Federal da Para\'iba, 58059-900, Jo\~ao Pessoa, PB, Brazil}{}
	\email{araujo@mat.ufpb.br}
	
	\author[R. Teymurazyan]{Rafayel Teymurazyan}
	\address{CMUC, University of Coimbra, CMUC, Department of Mathematics, 3001-501 Coimbra, Portugal}{}
	\email{rafayel@utexas.edu}
	
	\author[V. Voskanyan]{Vardan Voskanyan}
	\address{CMUC, University of Coimbra, CMUC, Department of Mathematics, 3001-501 Coimbra, Portugal}{}
	\email{vardan.voskanyan@mat.uc.pt}
	
	\begin{abstract}	
		We obtain sharp local $C^{1,\alpha}$ regularity of solutions for singular obstacle problems, Euler-Lagrange equation of which is given by
		$$
		\Delta_p u=\gamma(u-\varphi)^{\gamma-1}\,\text{ in }\,\{u>\varphi\},
		$$
		for $0<\gamma<1$ and $p\ge2$. At the free boundary $\partial\{u>\varphi\}$, we prove optimal $C^{1,\tau}$ regularity of solutions, with $\tau$ given explicitly in terms of $p$, $\gamma$ and smoothness of $\varphi$, which is new even in the linear setting.
		
		\bigskip
		
		
		\bigskip
		
	\end{abstract}
	
	\pagenumbering{arabic} 
	
	\maketitle
	
	\section{Introduction} \label{sct intro}
	
	In this paper we study minimization problems with non-differentiable zero order dependence. More precisely, in a bounded domain $\Omega\subset\mathbb{R}^n$, for a constant $\gamma\in(0,1)$, we study regularity of minimizers of
	\begin{equation}\label{1.1}
		J(u)=\inf_{v\in\mathbb{K}}J(v),
	\end{equation}
	where
	$$
	J(v):=\int_{\Omega}\left(\frac{|\nabla v|^p}{p}+(v-\varphi)^\gamma\right)\,dx,
	$$
	and
	\begin{equation}\label{minimizationset}
		\mathbb{K}:=\left\{v\in W^{1,p}(\Omega);\,\, v\geq\varphi,\,\,v-g\in W^{1,p}_0(\Omega)\right\},
	\end{equation}
	with $\varphi\in C^{1,\beta}(\Omega)$, for a $\beta\in(0,1]$ and $g\in W^{1,p}(\Omega)$. The corresponding Euler-Lagrange equation is 
	\begin{equation}\label{EL}
		\Delta_p u =\gamma(u-\varphi)^{\gamma-1} \quad \mbox{in } \; \{u>\varphi\} \cap \Omega,
	\end{equation}
	where
	$$
	\Delta_p u:=\Div\left(|\nabla u|^{p-2}\nabla u\right)
	$$
	is the standard $p$-Laplacian operator, $2\leq p<+\infty$. Note that the right hand side in \eqref{EL} blows up at the free boundary points
	$$
	\partial\{u>\varphi\}\cap\Omega,
	$$
	which makes it essential to understand its effect on the regularity of minimizers. The parameter $\gamma$, thus, measures the magnitude of singularity.
	
	Problems like \eqref{1.1} are used, for example, to model the density of a certain chemical in reaction with a porous catalyst pellet (see, for instance, \cite{A75}), and due to their wide range of applications, were studied by many prominent mathematicians. In the linear setting ($p=2$), the extreme cases ($\gamma=0$ and $\gamma=1$) of \eqref{1.1} were studied in \cite{AC81} and \cite{C80} with flat obstacles ($\varphi\equiv0$). The case $\gamma=0$ is related to jets flow and cavity problems, and minimizers are known to have local Lipschitz (optimal) regularity, as is established by Alt and Caffarelli in \cite{AC81}. In the nonlinear setting the extreme case of $\gamma=0$ was studied in \cite{DP05}, where Lipschitz regularity of minimizers is established. These type of problems, often referred to as Bernoulli type problems, appear in heat flows, \cite{A77}, electrochemical machining, \cite{LS87}, etc. The case of $\gamma=1$ resembles the classical obstacle problem, and its solution, as is shown by Brezis and Kinderlehrer in \cite{BK74}, is of class $C^{1,1}$,  $p=2$. In the nonlinear setting, $p>2$, the obstacle problem was studied in \cite{ALS,FKR17}. Its unique solution, as is shown in \cite{ALS}, is of class $C^{1,\alpha}_{\loc}$ at the free boundary points with
	$$
	\alpha=\min\left\{\beta,\frac{1}{p-1}\right\}.
	$$
	The problem \eqref{1.1} was studied in \cite{P1} (with $p=2$ and flat obstacle), where, using a minimizer preserving scaling, it was shown that minimizers are locally of the class $C^{1,\frac{\gamma}{2-\gamma}}$, $0<\gamma<1$ (see also \cite{AS22} and \cite{AT13}, for the problem governed by the infinity Laplacian and uniformly elliptic fully nonlinear operators, respectively). There are, so called, monotonicity formulas available in the linear case, which play a crucial role in the study of the problem. For $\gamma\in(0,1)$, the nonlinear case is covered in \cite{LQT}, where for obstacle type problem (with zero obstacle) is proved that minimizers are locally of the class $C^{1,\alpha}$ with
	\begin{equation}\label{1.3}
		\alpha=\min\left\{\sigma^-,\frac{\gamma}{p-\gamma}\right\},
	\end{equation}   
	where $\sigma^-$ is the H\"older regularity exponent for the gradient of $p$-harmonic functions ($a^-$ stands for any $b<a$). Actually, for $n=2$, from \cite{ATU1} one concludes that $\alpha=\frac{\gamma}{p-\gamma}$. Observe that in all the above results (except in \cite{ALS}) the obstacle is assumed to be trivial, guaranteeing a vanishing gradient of solutions at the free boundary points, which is essential in the analysis (in \cite{ALS} obstacle is assumed to be of the class $C^{1,\beta}$, but the zero order dependence of the functional is smooth). The methods, used to obtain those results, fail to work in the presence of non-trivial obstacles with large gradients at the free boundary, and a new approach is required to tackle the issue.
	
	\medskip
	
	In this work, we prove sharp regularity for minimizers of \eqref{1.1} both locally and at the free boundary points. More precisely, we show that minimizers are locally of the class $C^{1,\alpha}$,
	where
	\begin{equation}\label{alpha}
		\alpha=\min\left\{\sigma^-,\frac{\gamma}{p-\gamma},\frac{\beta}{p-1}\right\},
	\end{equation}
	and $\sigma>0$ is the H\"older regularity exponent of the gradient of $p$-harmonic functions. Note that our result extends the local regularity \eqref{1.3} of \cite{LQT} for problems with non-trivial obstacles, which is new even for the linear case ($p=2$). Moreover, at the free boundary points we obtain optimal $C^{1,\tau}$ regularity for minimizers of \eqref{1.1}, which, unlike local interior estimates, \textit{does not} depend on the regularity of $p$-harmonic functions. To be exact,
	we show that minimizers of \eqref{1.1} at the free boundary are in $C^{1,\tau}$, where
	\begin{equation}\label{1.4}
		\tau=\min\left\{\beta,\frac{\gamma}{p-\gamma}\right\},
	\end{equation}
	which generalizes the \textit{optimal} regularity result obtained in \cite{P1} for the linear case and trivial obstacle. Thus, at free boundary points the interior regularity result of \cite{LQT} improves substantially. Indeed, as the obstacle in \cite{LQT} is assumed to be trivial, then from \eqref{1.4}, we have
	\begin{equation*}
		\tau=\frac{\gamma}{p-\gamma},
	\end{equation*}
	which is better than $\alpha$ from \eqref{1.3}. Observe also, that our result extends (continuously) the optimal regularity result of \cite{ALS}, from smooth lower order dependence to the singular setting (Theorem \ref{mainresult}). 
	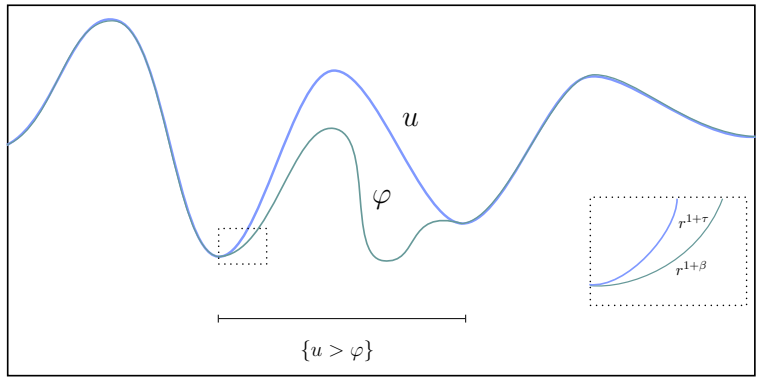
\begin{figure}[h]
		\begin{center}
			\scalebox{0.55} 
			{
				\begin{pspicture}(0,-5.5)(18.135412,3.5)
					\definecolor{colour0}{rgb}{0.5019608,0.6,1.0}
					\definecolor{colour1}{rgb}{0.4,0.6,0.6}
					\psframe[linecolor=black, linewidth=0.04, dimen=outer](18.123417,3.5)(0.023416443,-5.5)
					\psbezier[linecolor=colour0, linewidth=0.06](0.023416443,0.1)(1.2234164,0.7)(1.6234164,3.5)(2.7234163,3.1)(3.8234165,2.7)(4.128219,-2.5021117)(5.1234164,-2.6)(6.1186137,-2.6978884)(6.9237914,1.8726131)(7.9234166,1.9)(8.923041,1.927387)(10.023972,-1.7666851)(11.0234165,-1.8)(12.0228615,-1.8333148)(13.014926,1.2820944)(13.923416,1.7)(14.831907,2.1179056)(16.623417,0.2)(18.123417,0.3)
					\psbezier[linecolor=colour1, linewidth=0.04](5.1234164,-2.6)(6.4234166,-2.6)(6.933467,0.6414214)(7.9234166,0.5)(8.913366,0.35857865)(8.223416,-2.6)(9.123417,-2.7)(10.0234165,-2.8)(9.623417,-1.4)(11.0234165,-1.8)
					\psbezier[linecolor=colour1, linewidth=0.032](11.0234165,-1.8)(12.123417,-1.7)(13.223416,1.8)(14.223416,1.8)(15.223416,1.8)(16.423416,0.4)(18.123417,0.3)
					\psbezier[linecolor=colour1, linewidth=0.03](5.1234164,-2.6)(4.0234165,-2.4)(3.8234165,3.0)(2.6234164,3.1)(1.4234165,3.2)(1.2234164,0.6)(0.023416443,0.1)
					\psframe[linecolor=black, linewidth=0.04, linestyle=dotted, dotsep=0.10583334cm, dimen=outer](6.323416,-1.9)(5.1234164,-2.8)
					\psframe[linecolor=black, linewidth=0.04, linestyle=dotted, dotsep=0.10583334cm, dimen=outer](17.923416,-1.1434783)(14.106025,-3.8)
					\psbezier[linecolor=colour1, linewidth=0.03](14.123417,-3.3)(15.0234165,-3.4)(16.723417,-2.8)(17.323416,-1.2)
					\psbezier[linecolor=colour0, linewidth=0.04](14.101194,-3.2777777)(15.001194,-3.3777778)(16.223417,-2.0)(16.223417,-1.2)
					\rput[bl](16.190083,-3.0333333){$r^{1+\beta}$}
					\rput[bl](16.25675,-1.9){$r^{1+\tau}$}
					\rput[bl](8.85675,-1.4428571){{\huge $\varphi$}}
					\rput[bl](9.575797,0.56666666){{\huge $u$}}
					\rput[bl](7.124434,-5.124542){{\Large $\left\{ u > \varphi \right\}$}}
					\psline[linecolor=black, linewidth=0.02, tbarsize=0.07055555cm 6.0]{|-|}(5.1234164,-4.1)(11.123417,-4.1)
				\end{pspicture}
			}
			\caption{Detachment of $u$ from $\varphi$ at the free boundary.}
		\end{center}
	\end{figure}
	Our approach is based on geometric tangential analysis and a fine perturbation combined with adjusted scaling argument. Strictly speaking, we redeem regularity by ``tangentially accessing'' the information available in the ``flatness regime'' (for a rather comprehensive introduction to geometric tangential analysis, we refer the reader to \cite{TU21}). In the complementary case, i.e., when the gradient of a solution is bounded from below at a free boundary point, we use an ``adjusted scaling'' argument to ensure that in the limit we get a linear elliptic equation without the zero order term. The idea of the adjustment is to get rid of those terms that blow up at the limit.
	
	\medskip
	
	The paper is organized as follows: in Section \ref{s3}, we prove existence of minimizers and in Section \ref{s5}, establish local sharp $C^{1,\alpha}$ regularity result (Theorem \ref{t3.1}). In Section \ref{s6}, we obtain optimal regularity at free boundary points for minimizers with small gradient (Theorem \ref{t4.1}). Section \ref{s8} is devoted to the adjusted scaling argument. Finally, in Section \ref{s9}, we obtain sharp regularity for minimizers with large gradient at the free boundary points (Theorem \ref{t5.1}). We close the paper with two appendices, containing some auxiliary technical results (Appendix \ref{A}) and a list of several known ones (Appendix \ref{B}), that are used in the paper. 
	
	\subsection*{Notations and assumptions}
	Hereafter $B_r(x_0)$ is the ball of radius $r$ centered at $x_0$, $B_r(0)=B_r$, and $|B_r|$ stands for the volume of the ball $B_r$. When $\Omega=B_r$ in \eqref{minimizationset}, we will often use the notation $\mathbb{K}_r$ instead of $\mathbb{K}$. Additionally, for a given integrable function $f$, we denote by $(f)_{r}$ its average on the ball of radius $r$ centered at the origin, i.e.,  
	$$
	(f)_{r}:=\frac{1}{|B_r|}\int_{B_r} f(x)\,dx.
	$$
	To avoid repetition of arguments when applying the conclusions for different set of functions, we introduce a function $H:\R^n\to\R_+$, which is assumed to be of class $C^2$ and satisfy the following structural assumptions
	\begin{equation}\label{assumptions}
		\left\{
		\begin{array}{rclcl}
			|\nabla H(\xi)|&\le&\displaystyle \Upsilon\,\omega(|\xi|), \vspace*{0.2cm}\\
			|D^2 H(\xi)|&\le&\displaystyle\Lambda \frac{\omega(|\xi|)}{|\xi|},\vspace*{0.2cm}\\
			\eta^T D^2 H(\xi)\eta&\geq&\displaystyle\lambda \frac{\omega(|\xi|)}{|\xi|}|\eta|^2,
		\end{array}
		\right.
	\end{equation}
	with
	$$
	\omega(z):=\kappa_1z^{p-1}+\kappa_2z,\,\,\,z\ge0,
	$$
	where $\xi$, $\eta\in\R^n$ and $\Upsilon>0$, $\Lambda\ge\lambda>0$, $\kappa_1\ge0$, $\kappa_2\ge0$ are constants, and $\kappa_1+\kappa_2>0$. We also will use the notation
	\begin{equation}\label{3.1}
		G(z):=\int_0^z\omega(\zeta)\,d\zeta=\kappa_1\frac{z^p}{p}+\kappa_2\frac{z^2}{2},\quad z\in\R_+.
	\end{equation}
	\begin{remark}\label{r2.1}
		The function $H(\xi)=p^{-1}|\xi|^p$ satisfies the above conditions with $\kappa_1=1$ and $\kappa_2=0$. The classical linear version, $p=2$, is recovered by assuming $\kappa_1=0$. An alternative example of a function $H$ satisfying \eqref{assumptions} is constructed in Appendix \ref{A}.
	\end{remark}
	
	\section{Existence of minimizers}\label{s3}
	In this section we show that there exists at least one minimizer of \eqref{1.1}. Unlike the regular case ($\gamma$=1), which is known to have a unique minimizer (see, for example, \cite{ALS,CLRT14,FKR17,PSU12,RT11}) in the singular setting this is not assured.
	\begin{theorem}\label{existence}
		If $\varphi\in W^{1,p}(\Omega)\cap L^\infty(\Omega)$, $g\in W^{1,p}(\Omega)$ and $\gamma\in(0,1)$, then there exists a minimizer $u$ of \eqref{1.1}. Moreover, 
		\begin{equation}\label{bound}
			\|u\|_{L^\infty(\Omega)}\le\max\left\{\|g\|_{L^{\infty}(\Omega)},\|\varphi\|_{L^\infty(\Omega)}\right\}.
		\end{equation}
	\end{theorem}
	\begin{proof}
		Set
		$$
		m:=\inf_{v\in\mathbb{K}}J(v)\ge0.
		$$
		If $v_i\in\mathbb{K}$ is a minimizing sequence, then for $i\ge i_0$, $i_0\in\mathbb{N}$, one has 
		$$
		0\le J(v_i)\le m+1,
		$$
		hence
		$$
		\int_\Omega|\nabla v_i|^p\,dx\le pJ(v_i)\le p(m+1),
		$$
		and the Poincar\'e inequality yields that the sequence $v_i$ is bounded in $W_0^{1,p}(\Omega)$. Therefore, by Rellich-Kondrachov theorem, there is a function $u\in W_0^{1,p}(\Omega)$ such that, up to a subsequence,
		$$
		v_i\to u\, \text{ weakly in }\, W^{1,p}(\Omega)\,\,\,\text{ and }\,\,\, v_i\to u\, \text{ in }\,L^p(\Omega).
		$$
		Notice that $u\in\mathbb{K}$, and thus, using the lower semi-continuity of the Dirichlet integral, we obtain
		$$
		m\le J(u)\le\liminf_{j\to\infty}J(v_i)=m,
		$$
		i.e., $u$ is a minimizer of \eqref{1.1}.
		
		To see \eqref{bound}, set
		$$
		u^M:= \min\{u, M\}\in\mathbb{K},
		$$
		where
		$$
		M:=\max\left\{\|g\|_{L^{\infty}(\Omega)},\|\varphi\|_{L^\infty(\Omega)}\right\}.
		$$	
		Since $u$ is a minimizer,
		$$
		\int_{\Omega}\frac{|\nabla u|^p}{p}\,dx-\int_{\Omega}\frac{|\nabla u^M|}{p}\,dx\le\int_{\Omega} \left((u^M-\varphi)^{\gamma}-(u-\varphi)^{\gamma}\right)\,dx,
		$$
		and thus
		$$
		0\leq \int_{\{u>M\}}\frac{|\nabla u|^p}{p}\,dx \le\int_{\Omega} \left((u^M-\varphi)^{\gamma}-(u-\varphi)^{\gamma}\right)\,dx\le0,
		$$
		which implies that $u=u^M$. Therefore, $-\|\varphi\|_{L^\infty(\Omega)}\le\varphi\le u\le M$.
	\end{proof}
	
	\section{Local $C^{1,\alpha}$ regularity estimates}\label{s5}
	
	One of the main steps towards the optimal regularity of minimizers is obtaining $C^{1,\alpha}$ regularity for minimizers of
	\begin{equation}\label{J}
		J_\delta(v):=\int_\Omega\left(H(\nabla v)+\delta(v-\varphi)^\gamma\right)\,dx
	\end{equation}
	over the set $\mathbb{K}$, defined by \eqref{minimizationset}. Here $\delta\in[0,1]$, $p\ge2$ and $H$ satisfies \eqref{assumptions}. The key step towards the regularity is the result on the decay of integral oscillation -- comparing energy estimates involving minimizers of $J_\delta$ with the ones of
	$$
	\min\limits_{v \in W^{1,p}(\Omega)} \int_\Omega\left( H(\nabla v)+ H(\nabla \varphi) \cdot \nabla v \right) \, dx, 
	$$
	where the latter is an unconstrained (and non-singular) minimization problem with smooth first order coefficients. Local sharp regularity for minimizers of $J_\delta$ is established based on the following auxiliary lemmas.
	\begin{lemma}\label{l3.4}
		If $f\in C^{0,\beta}(B_R;\R^n)$ for some $\beta\in(0,1)$, and $w$ is a minimizer of
		$$
		\int_{B_R}\left(H(\nabla w)-f\cdot\nabla w\right)\,dx,
		$$
		in $W^{1,p}_g(B_R)$, then there exist constants $C$, $\sigma >0$ depending only on $\kappa_1$, $\kappa_2$, $\|f\|_{C^{0, \beta}(B_R)}$ and  $ \|w\|_{L^{\infty}(B_R)}$, such that
		\begin{equation*}
			\int_{B_r} G(|\nabla w-(\nabla w)_{r}|)\,dx\le C \left(\frac{r}{R}\right)^{n+q\sigma}\int_{B_R}G(|\nabla w-(\nabla w)_{R}|)\,dx+CR^{n+q\frac{\beta}{q-1}},
		\end{equation*}
		where
		\begin{equation}\label{q}
			q:=\left\{
			\begin{array}{ccl}
				2, & \mbox{when} & \kappa_1=0,\\[0.2cm] 
				p, & \mbox{when} & \kappa_1>0.
			\end{array}
			\right.
		\end{equation}
	\end{lemma}
	\begin{proof}
		Let $v_R$ be the minimizer of 
		$$
		\int_{B_R} H(\nabla v)\,dx,
		$$
		in $W_w^{1,p}(B_R)$. From \eqref{leq:HPQ} of Appendix \ref{A}, we have
		\begin{equation}\label{3.5}
			\begin{aligned}
				&\int_{B_R}\left(H(\nabla w)-H(\nabla v_R) \right)\,dx\\
				&\geq\int_{B_R}\nabla H(\nabla v_R)\cdot(\nabla w-\nabla v_R)\,dx+c\int_{B_R} G(|\nabla w- \nabla v_R|)\,dx,
			\end{aligned}
		\end{equation}
		where $c>0$ is a universal constant. In addition, as
		$$
		\varphi(t):=\int_{B_R} H(\nabla v_R +t\nabla (w-v_R))\, dx
		$$	
		has a minimum at $t=0$, then
		$$
		\int_{B_R} \nabla H(\nabla v_R)\cdot \nabla (w-v_R) \, dx=\frac{d}{dt} \varphi(t)\Big|_{t=0}\geq0,
		$$
		which, combined with \eqref{3.5}, provides
		\begin{equation}\label{3.6}
			\displaystyle\int_{B_R}\left(H(\nabla w)-H(\nabla v_R)\right)\,dx\geq  c\displaystyle\int_{B_R} G(|\nabla w- \nabla v_R|)\,dx.
		\end{equation}
		On the other hand, the definition of $w$ implies \begin{equation}\label{3.7}
			\int_{B_R} f\cdot (\nabla w - \nabla v_R) \; dx \geq \int_{B_R}\left[H(\nabla w) - H(\nabla v_R)\right]\,dx.    
		\end{equation}
		From \eqref{3.6} and \eqref{3.7} one has
		\begin{equation}\label{3.8}
			\int_{B_R} G(|\nabla w- \nabla v_R|)\,dx\le\frac{1}{c}\int_{B_R} f\cdot(\nabla w-\nabla v_R)\,dx.
		\end{equation}
		Observe that as $v_R=w$ on $\partial B_R$, then
		$$
		\int_{B_R}(f)_{R}\cdot (\nabla w-\nabla v_R)\,dx=0,
		$$
		which, together with \eqref{3.8} and the H\"older inequality gives
		\begin{equation}\label{3.9}
			\begin{array}{rl}
				&\displaystyle\int_{B_R} G(|\nabla w- \nabla v_R|)\,dx\leq  \displaystyle\frac{1}{c} \displaystyle\int_{B_R}(f-(f)_{R})\cdot (\nabla w -\nabla v_R)\,dx \\[0.5cm]
				\leq & \displaystyle\frac{1}{c}\left(\int_{B_R}|f-(f)_{R}|^{q'}\,dx \right)^{\frac{1}{q'}}\left(\displaystyle\int_{B_R} |\nabla w-\nabla v_R|^q\,dx \right)^{\frac{1}{q}},
			\end{array}
		\end{equation}
		where $q'$ is the conjugate of $q$, i.e., $q'=\frac{q}{q-1}$. Recalling that $f \in C^{0,\beta}$ and using Campanato's characterization of H\"older spaces,  \cite[Theorem 5.5]{GM12}, from \eqref{3.9} we deduce
		\begin{equation}\label{3.10}	
			\int_{B_R} G(|\nabla w-\nabla v_R|)\,dx\leq  C R^{\frac{n(q-1)+q\beta}{q}}\left(\displaystyle\int_{B_R} |\nabla w -\nabla v_R|^q\,dx \right)^{\frac{1}{q}},	
		\end{equation}
		where $C>0$ is a universal constant. Observe that (see \eqref{3.1})
		$$
		\int_{B_R} |\nabla w -\nabla v_R|^q\,dx\le\int_{B_R} G(|\nabla w-\nabla v_R|)\,dx,
		$$
		therefore, \eqref{3.10} leads to
		\begin{equation}\label{3.11}
			\displaystyle\int_{B_R} G(|\nabla w-\nabla v_R|)\,dx\leq C R^{n+q\frac{\beta}{q-1}},
		\end{equation}
		for a universal constant $C>0$, depending only on $\kappa_1$, $\kappa_2$ and $p$. Combining the latter with Lemma \ref{l3.3} from Appendix \ref{A}, for $r\in(0,R)$, we obtain
		\begin{equation}\label{3.12}
			\begin{array}{ccl}
				\displaystyle\int_{B_r} G(|\nabla w-(\nabla w)_{r}|)\,dx &\leq& \displaystyle C\int_{B_r} G(|\nabla v_R-(\nabla v_R)_{r}|)\,dx \\[0.5cm] 
				&+&CR^{n+q\frac{\beta}{q-1}}.
			\end{array}
		\end{equation} 
		We estimate the first term of the right hand side in \eqref{3.12} by applying Lemma \ref{l3.2} from Appendix \ref{A} to $v_R$:
		\begin{equation}\nonumber	
			\int_{B_r}G(|\nabla v-(\nabla v)_{r}|)\,dx\leq C\left(\frac{r}{R}\right)^{n+q\sigma} \int\limits_{B_R} G(|\nabla v_R-(\nabla v_R)_{R}|)\,dx,
		\end{equation}
		for a universal constant $\sigma>0$. We then estimate the right hand side of the last inequality by using Lemma \ref{l3.3} together \eqref{3.11} to arrive at  
		\begin{equation}\nonumber
			\displaystyle\int\limits_{B_R}G(|\nabla v_R-(\nabla v_R)_{R}|)\,dx\le C\int\limits_{B_R} G(|\nabla w-(\nabla w)_{R}|)\,dx+CR^{n+q\frac{\beta}{q-1}}.
		\end{equation}
		Plugging the last two inequalities into \eqref{3.12}, we obtain the desired result.
	\end{proof}
	Next, using the unconstrained problem of Lemma \ref{l3.4}, we prove a gradient integral oscillation decay for minimizers of the constrained $H$-Dirichlet energy. Its proof is based on auxiliary lemmas from Appendix \ref{A}.
	\begin{lemma}\label{l3.5}
		If $\varphi \in C^{1,\beta}(B_R)$ and $u$ is the solution of the obstacle problem
		\begin{equation}\label{3.13}
			\int_{B_R}H(\nabla u)\,dx=\min_{v\in\mathbb{K}_R}\int_{B_R}H(\nabla v)\,dx,
		\end{equation}
		where $\mathbb{K}_R$ is defined by \eqref{minimizationset}, then there exist $C>0$ and $\sigma>0$ constants, depending only on $\kappa_1$, $\kappa_2$, $\|\varphi\|_{C^{1, \beta}(B_R)}$ and  $\|u\|_{L^{\infty}(B_R)}$, such that for $r\in(0,R)$,
		\begin{equation*}\label{3.15}
			\begin{aligned}
				\int_{B_r} G\left(|\nabla u-(\nabla u)_{r}|\right)\,dx
				&\leq C \left(\frac{r}{R}\right)^{n+q\sigma}\int_{B_R} G(|\nabla u-(\nabla u)_{R}|)\,dx\\
				&+CR^{n+q\frac{\beta}{q-1}},
			\end{aligned}
		\end{equation*}
		where $q$ is defined by \eqref{q}.	
	\end{lemma}
	\begin{proof}
		If $w_R$ is the minimizer of 
		$$
		\int_{B_R}\left(H(\nabla w)-\nabla H(\nabla \varphi)\cdot\nabla w\right)\,dx
		$$ 
		in $W^{1,p}_u(B_R)$, then 
		\begin{equation}\label{3.16}
			\Div(\nabla H(\nabla w_R))=\Div(\nabla H(\nabla \varphi))\quad\textrm{ in }\quad B_R,
		\end{equation}
		with $w_R\geq\varphi$ on $\partial B_R$. By the maximum principle (see, for example, \cite{HKM06}), $w_R\geq\varphi$ in $B_R$. Thus, $w_R\in\mathbb{K}_R$, i.e., it is a competing function in the obstacle problem \eqref{3.13}, therefore,
		$$
		\int_{B_R}\left(H(\nabla u)-H(\nabla w_R)\right)\,dx\leq 0.
		$$
		On the other hand, recalling \eqref{leq:HPQ} from Appendix \ref{A}, we have
		\begin{equation*}
			\begin{aligned}
				\int_{B_R}\left(H(\nabla u)-H(\nabla w_R)\right)\,dx
				&\geq\int_{B_R}\nabla H(\nabla w_R)\cdot(\nabla u -\nabla w_R)\,dx\\
				&+c\int_{B_R}G(|\nabla u-\nabla w_R|)\,dx,
			\end{aligned}
		\end{equation*}
		hence
		\begin{equation}\label{3.17}
			\begin{aligned}
				c\int_{B_R}G(|\nabla u-\nabla w_R|)\,dx &\leq\int_{B_R}\nabla H(\nabla w_R)\cdot(\nabla w_R-\nabla u)\,dx\\
				&=\int_{B_R}\nabla H(\nabla\varphi)\cdot(\nabla w_R-\nabla u)\,dx,
			\end{aligned}
		\end{equation}
		where the equality is obtained as a consequence of \eqref{3.16}. As in the proof of Lemma \ref{l3.4}, using Campanato's characterization of the H\"older continuity for $f:=\nabla H(\nabla\varphi)$ and the H\"older inequality, from \eqref{3.17}, we deduce
		\begin{equation*}
			\begin{aligned}
				\int_{B_R}G(|\nabla u-\nabla w_R|)\,dx	&\leq C\int_{B_R}\left(f-(f)_{R}\right)\cdot(\nabla w_R-\nabla u)\,dx\\
				&\leq C\left(\int_{B_R} |\nabla u-\nabla w_R|^q\,dx\right)^{\frac{1}{q}} R^{\frac{n(q-1)}{q}+\beta},
			\end{aligned}
		\end{equation*}
		which, as
		$$
		\int_{B_R} |\nabla u-\nabla w_R|^q\,dx\le\int_{B_R}G(|\nabla u-\nabla w_R|)\,dx,
		$$ 
		provides
		$$
		\int_{B_R} G(|\nabla u-\nabla w_R|)\,dx\leq C R^{n+q\frac{\beta}{q-1}},
		$$
		where $C>0$ is a universal constant. The latter, combined with Lemma \ref{l3.2} and Lemma \ref{l3.3} from Appendix \ref{A}, gives the desired result, as argued in the proof of Lemma \ref{l3.4}.
	\end{proof}
	We are now ready to prove the main result of this section. It follows by using the lemmas obtained above and invoking arguments similar to those in \cite[Theorem 2]{C91} and \cite[Theorem 1.1]{LQT}. 
	\begin{theorem}\label{t3.1}
		If $u$ is a minimizer of 
		$$
		J_\delta(u)=\inf_{v\in\mathbb{K}}J_\delta(v),
		$$
		where $J_\delta$ is defined by \eqref{J} and $\mathbb{K}$ is defined by \eqref{minimizationset}, then for every $\Omega'\subset\subset\Omega$, there exists a constant $C>0$, depending only on $\dist(\Omega',\partial\Omega)$, $\|u\|_{L^{\infty}(\Omega)}$, $\|\varphi\|_{C^{1,\beta}(\Omega)}$, $p$, $\gamma$ and $n$, but not depending on $\delta$, such that
		$$
		\|u\|_{C^{1,\alpha}(\Omega')} \leq C,
		$$
		for an $\alpha\in(0,1)$ depending only on $p$, $\gamma$, $\beta$ and $n$ and
		$$
		\alpha=\min\left\{\sigma^-,\frac{\gamma}{q-\gamma},\frac{\beta}{q-1}\right\},
		$$
		where $q$ is defined by \eqref{q}, and $\sigma>0$ is the H\"older regularity exponent for the gradient of $H$-harmonic functions.
	\end{theorem}
	\begin{proof}
		Without loss of generality we may assume $B_R\subset\Omega$ for some $R>0$. If $h=h_R$ is the unique solution of the $H$-obstacle problem (see, for example, \cite{CLRT14, RT11})
		$$
		\min_{v\in\mathbb{K}_R}\int_{B_R}H(\nabla v)\,dx,
		$$
		where $\mathbb{K}_R$ is defined by \eqref{minimizationset}, then 
		$$
		\int_{B_R} H(\nabla h +t\nabla (u-h))\,dx
		$$	
		has minimum at $t=0$, therefore,
		$$
		\int_{B_R}\nabla H(\nabla h)\cdot\nabla (u-h)\,dx=\frac{d}{dt}\int_{B_R}H(\nabla h +t\nabla(u-h))\,dx\Big|_{t=0}\geq 0.
		$$
		The latter, combined with \eqref{leq:HPQ}, provides
		\begin{equation*}\label{3.18}
			\begin{aligned}
				c\int_{B_R}|\nabla u-\nabla h|^{q}\,dx&\leq c\int_{B_R}G(|\nabla u-\nabla h|)\,dx\\
				&\leq\int_{B_R}\left(H(\nabla u)-H(\nabla h)\right)\,dx,
			\end{aligned}
		\end{equation*}
		where $q$ is defined by \eqref{q}. On the other hand, since $u$ is a minimizer of $J_\delta$, using H\"older and Poincar\'e inequalities, we obtain
		\begin{equation*}
			\begin{aligned}
				\int_{B_R}\left(H(\nabla u)-H(\nabla h)\right)\,dx&\leq\delta \int_{B_R}\left[(h-\varphi)^\gamma- (u-\varphi)^\gamma\right]\,dx\\
				&\leq\delta\int_{B_R}|u-h|^\gamma\,dx\\
				&\leq C|B_R|^{1-\frac{\gamma}{q^*}} \left(\int_{B_R} |u-h|^{q^*}\,dx\right)^{\frac{\gamma}{q^*}}\\
				&\leq  C|B_R|^{1-\frac{\gamma}{q^*}}\left(\int_{B_R}|\nabla u-\nabla h|^{q}\,dx\right)^{\frac{\gamma}{q}},
			\end{aligned}
		\end{equation*}
		where 
		$$
		\frac{1}{q^*}=\frac{1}{q}-\frac{1}{n}.
		$$
		Then \eqref{3.17}, coupled with the last inequality, yields
		\begin{equation}\label{3.19}
			\int_{B_R}|\nabla u-\nabla h|^{q}\,dx\le C|B_R|^{\frac{q(q^*-\gamma)}{q^*(q-\gamma)}}=CR^{n+ q\frac{\gamma}{q-\gamma}},
		\end{equation}
		where the constant $C>0$ depends only on $n$ and $p$. Once again, arguing as in the proof of Lemma \ref{l3.4} (with $\kappa_2=0$ and $\kappa_1>0$ and with $\kappa_2>0$ and $\kappa_1=0$), the last estimate implies
		\begin{equation*}
			\begin{aligned}
				\int_{B_R}|\nabla u-(\nabla u)_{R}|^{q}\,dx&\le\int_{B_R}|\nabla h-(\nabla h)_{R}|^{q}\,dx\\
				&+CR^{n+q\frac{\gamma}{q-\gamma}}.
			\end{aligned}
		\end{equation*}
		The latter, combined with Lemma \ref{l3.5}, for $r\in(0,R)$ provides, for a $\sigma>0$,
		\begin{equation}\label{3.20}
			\begin{aligned}
				\int_{B_r}|\nabla u-(\nabla u)_{r}|^{q}\,dx&\leq C\left(\frac{r}{R}\right)^{n+q\sigma} \int_{B_R}|\nabla h-(\nabla h)_{R}|^{q}\,dx\\
				&+CR^{n+q\frac{\beta}{q-1}}+CR^{n+q\frac{\gamma}{q-\gamma}}.
			\end{aligned}
		\end{equation}
		Making use of Lemma \ref{l3.3} from Appendix \ref{A} (with $\kappa_2=0$ and $\kappa_1>0$ and with $\kappa_2>0$ and $\kappa_1=0$), \eqref{3.20} and \eqref{3.19}, we arrive at
		\begin{equation*}\label{3.21}
			\begin{aligned}
				\int_{B_r}|\nabla u-(\nabla u)_{r}|^{q}\,dx&\leq C\left(\frac{r}{R}\right)^{n+q\sigma} \int_{B_R}|\nabla u-(\nabla u)_{R}|^{q}\,dx\\
				&+CR^{n+q\frac{\beta}{q-1}}+CR^{n+q\frac{\gamma}{q-\gamma}},
			\end{aligned}
		\end{equation*}
		which, in terms of the non-negative and non-decreasing function
		$$
		\phi(r):=\int_{B_r}|\nabla u-(\nabla u)_{r}|^{q}\,dx,
		$$
		takes the form
		\begin{equation*}
			\phi(r)\le 
			C\left(\frac{r}{R}\right)^{n+q \sigma}\phi(R)+CR^{n+q\frac{\beta}{q-1}}+
			CR^{n+q\frac{\gamma}{q-\gamma}}.
		\end{equation*}
		Lemma \ref{l2.1} from Appendix \ref{B} then for any $r\le R$, guarantees
		$$
		\phi(r)\le 
		C\left[\left(\frac{r}{R}\right)^{n+q \alpha}\phi(R)+r^{n+q\alpha}\right],
		$$
		where $\alpha<\sigma$ and
		$$
		\alpha\le\min\left\{\frac{\gamma}{q-\gamma},\frac{\beta}{q-1}\right\}.
		$$
		Thus, for all $r\le\frac1 2\dist(0, \partial\Omega)$, we have
		$$
		\left(r^{-n-q\alpha}\int_{B_r}|\nabla u-(\nabla u)_{r}|^p\,dx\right)^{\frac{1}{q}} \leq C,
		$$
		where the constant $C>0$ depends only on $p$, $\alpha$, $\dist(0,\partial\Omega)$, $\|u\|_{L^{\infty}(\Omega)}$, $\|\varphi\|_{C^{1,\beta}(\Omega)}$ and $n$. The result then follows from Campanato's characterization of H\"older continuous functions (see, for example, \cite[Theorem 5.5]{GM12}).
	\end{proof}
	\begin{remark}\label{r5.1}
		Observe that for $H(\xi)=p^{-1}|\xi|^p$, the constant $\sigma>0$ is the H\"older exponent of the gradient for $p$-harmonic functions (see Lemma \ref{l3.2} in Appendix \ref{A}), and Theorem \ref{t3.1} reproduces the local regularity result \eqref{1.3} of \cite{LQT} for problems with non-trivial obstacles.
	\end{remark}
	
	\section{Small gradient estimates}\label{s6}
	Using Theorem \ref{t3.1}, we obtain sharp regularity at the free boundary points for minimizers of \eqref{1.1}. We distinguish two cases: when the gradient of a minimizer is relatively small and when its large. This section is devoted to the analysis of the first case. Observe that at the free boundary points gradient of the solution and that of the obstacle are equal, since $u-\varphi$ admits minimum at those points. This emphasizes that the case of zero obstacle (studied in \cite{P1}) cannot be adapted to work for general obstacles that may have large gradient at the free boundary. Our approach, however, makes use of geometric tangential analysis methods - leading to sharp regularity in a broader framework. 
	
	\medskip
	We start by observing that from Theorem \ref{t3.1}, applied for $H(\xi):=p^{-1}|\xi|^p$, we know that solutions of
	\begin{equation}\label{4.1}
		I_\delta(u)=\min_{v\in\mathbb{K}_1}I_\delta(v),
	\end{equation}
	where
	\begin{equation}\label{4.2}
		I_\delta(v):=\int_{B_1}\left(\frac{|\nabla v|^p}{p}+\delta(v-\varphi)^\gamma\right)\,dx,
	\end{equation}
	and $\mathbb{K}_1$ is defined by \eqref{minimizationset}, are of class $C_{\loc}^{1,\alpha}$ (uniform in  $\delta$), where $\alpha$ is given explicitly by \eqref{alpha}. As above, the function $\varphi\in C^{1,\beta}(B_1)$, i.e., there exists a universal constant $L>0$ such that
	\begin{equation}\label{L}
		\|\varphi\|_{C^{1,\beta}(B_{r/2})}\le L\|\varphi\|_{L^\infty(B_r)},
	\end{equation}
	for any $0<r\le1$. 
	
	\medskip
	The following two lemmas enable the ``tangential access'' to the sharp regularity information for the classical $p$-obstacle problem. To proceed, we define the ``flatness'' constant
	\begin{equation}\label{flatness}
		\mu:=\frac{1}{2^{1+\beta}C_1C_2L},
	\end{equation}
	where $C_1>0$ is the constant from Theorem \ref{t2.3} of Appendix \ref{B}, $C_2>0$ is the constant from Lemma \ref{l2.2} (both of them depend only on $p$ and $n$), $\beta\in(0,1]$ is the H\"older regularity exponent of $\nabla\varphi$, and $L>0$ is the universal constant in \eqref{L}.
	\begin{lemma}\label{l4.1}
		Let $u$ be a minimizer of \eqref{4.1}, $\|u\|_{L^\infty(B_1)}\le1$, $\varphi\in C^{1,\beta}(B_1)$, for some $\beta\in(0,1]$, $\|\varphi\|_{L^\infty(B_1)}\le\mu$, where $\mu>0$ is defined by \eqref{flatness}, $\varphi(0)=0$ and $0\in\partial\{u>\varphi\}$. Then for a given $\varepsilon\in(0,\frac 1 4)$, there exists $\delta_\varepsilon>0$, such that whenever $0\le\delta\le\delta_\varepsilon$ and
		$$
		|\nabla\varphi(0)|\le\delta_\varepsilon,
		$$
		then
		$$
		\sup_{B_\varepsilon}|u|\le\varepsilon^{1+\beta}.
		$$
	\end{lemma}
	\begin{proof}
		We argue by contradiction and assume that the conclusion of the lemma fails to hold. Thus, we assume that there exist $\varepsilon_0>0$, $\delta=\delta_k$, minimizer $u_k$ of $I_{\delta_k}$ with obstacle $\varphi=\varphi_k\in C^{1,\beta}(B_1)$, such that
		$$
		0\in\partial\{u_k>\varphi_k\}, \quad \varphi_k(0)=0, \quad \|u_k\|_{L^\infty(B_1)} \leq 1,\quad\|\varphi_k\|_{L^\infty(B_1)}\le\mu
		$$
		and
		$$
		\delta_k\le1/k,\quad|\nabla \varphi_k(0)|\le 1/k,
		$$
		but
		\begin{equation*}\label{4.5}
			\sup_{B_{\varepsilon_0}}|u_k|>\varepsilon_0^{1+\beta}.
		\end{equation*}
		Theorem \ref{t3.1}, combined with the Arzel\`a-Ascoli theorem, guarantees the existence of functions $u_\infty$, $\varphi_\infty$, such that up to a sub-sequence, as $k\to\infty$,
		\begin{equation}\label{4.6}
			u_k\rightarrow u_{\infty}\,\textrm{ locally uniformly in }\,C^{1,\alpha}(B_1)
		\end{equation}
		and 
		$$
		\varphi_k\rightarrow \varphi_\infty\,\text{ locally uniformly in }\; C^{1,\beta}(B_1).
		$$
		Note that $u_{\infty}\ge \varphi_\infty$, $u_{\infty}(0)=\varphi_\infty(0)=|\nabla\varphi_\infty(0)|=0$, $\|\varphi_\infty\|_{L^\infty(B_1)}\le\mu$ and
		\begin{equation}\label{4.7}
			\sup_{B_{\varepsilon_0}}|u_{\infty}|>\varepsilon_0^{1+\beta}.
		\end{equation}
		Moreover, since
		$$
		I_{\delta_k}(u_k) \leq   I_{\delta_k}(u_k+\epsilon v) \quad \forall \epsilon>0,\,\,\,\forall v\in C_0^\infty(B_1), \,\,\,v\ge0,
		$$
		then, in view of \eqref{4.6}, as $\delta_k\to0$, one has
		$$
		I_{0}(u_\infty) \leq   I_{0}(u_\infty+\epsilon v),
		$$
		where
		$$
		I_0(w):=\int_{B_1}\frac{|\nabla w|^p}{p}\,dx,
		$$
		which yields 
		$$
		\int_{B_1}|\nabla u_{\infty}|^{p-2}\nabla u_{\infty}\cdot\nabla v\, dx \geq 0,  \quad \forall  \; v\ge 0, \; v\in C_0^{\infty}(B_{1}),
		$$
		i.e., $\Delta_p u_{\infty} \leq 0$ in $B_1$, that is, $u_{\infty}$ is $p$-superharmonic in $B_{1}$. The proof now follows invoking ideas similar to those used in \cite[Theorem 1]{ALS}. Setting $\lambda:=2\varepsilon_0$, we define
		$$
		u_*(x)=\frac{u_\infty(\lambda x)}{\lambda^{1+\beta}} 
		\quad \mbox{and} \quad
		\varphi_*(x)=\frac{\varphi_\infty(\lambda x)}{\lambda^{1+\beta}}\quad \mbox{in } \; B_1.
		$$
		Then $u_*$ is $p$-superharmonic and satisfies $u_*\ge\varphi_*$. Also, as $\varphi_\infty\in C^{1,\beta}(B_1)$, recalling \eqref{L}, we write
		\begin{equation*}
			\sup_{B_\lambda}|\varphi_\infty(x)-\varphi_\infty(0)-\nabla\varphi_\infty(0)\cdot x|\le L\|\varphi_\infty\|_{L^\infty(B_{2\lambda})}\lambda^{1+\beta}\le L\mu\lambda^{1+\beta}.
		\end{equation*}
		But  $\varphi_\infty(0)=|\nabla \varphi_\infty(0)|=0$, 
		and so
		$$
		\sup_{B_\lambda}|\varphi_\infty|\le\mu L\lambda^{1+\beta}.
		$$
		Therefore,	\begin{equation*}		
			\|\varphi_*\|_{L^\infty(B_1)}=\frac{\|\varphi_\infty\|_{L^\infty(B_\lambda)}}{\lambda^{1+\beta}}\le\mu L.
		\end{equation*}
		Thus,
		$$
		u_*+\mu L\geq -\|\varphi_*\|_{L^\infty(B_1)}+\mu L\geq0,
		$$
		and the weak Harnack inequality (Theorem \ref{t2.3} of Appendix \ref{B}), provides
		\begin{equation}\label{4.8}
			\begin{array}{rcl}
				\|u_*+\mu L\|_{L^s(B_{3/4})}&\le& C_1\displaystyle\inf_{B_{1/2}}\left(u_*+\mu L\right)\\
				&\le&C_1(\varphi_*(0)+\mu L)\\
				&=&C_1\mu L,
			\end{array}
		\end{equation}
		for some universal $s>1$. Here $C_1>0$ is a constant depending only on $p$ and $n$. Since the function
		$$
		w:=\max\left(u_*+\mu L,\,\sup_{B_1}\varphi_*+\mu L\right)
		$$
		is $p$-superharmonic in $B_1$, applying Lemma \ref{l2.2} of Appendix \ref{B}, we obtain
		\begin{equation}\label{4.9}
			\sup_{B_{1/2}}w\le C_2\|w\|_{L^s(B_{3/4})},
		\end{equation}
		where $C_2>0$ is a constant depending only on $p$ and $n$. 
		Combining \eqref{4.8} and \eqref{4.9}, we deduce
		\begin{equation}\label{4.10}
			\sup_{B_{1/2}}u_*\le\mu C_1C_2L
		\end{equation}
		As also $u_*\ge-\|\varphi_*\|_{L^\infty(B_1)}\ge-\mu L$, recalling \eqref{flatness}, from \eqref{4.10} we get
		$$
		\sup_{B_{1/2}}|u_*|\le\frac{1}{2^{1+\beta}}.
		$$
		Consequently,
		$$
		\sup_{B_{\varepsilon_0}}|u_\infty|\le\varepsilon_0^{1+\beta},
		$$
		which contradicts \eqref{4.7}.
	\end{proof}
	The next result provides a discrete version of the desired oscillation estimate.
	\begin{lemma}\label{l4.2}	
		Let $u$ be a minimizer of \eqref{4.1}, $\|u\|_{L^\infty(B_1)}\le1$, $\varphi\in C^{1,\beta}(B_1)$, for some $\beta\in(0,1]$, $\|\varphi\|_{L^\infty(B_1)}\le\frac{\mu}{2L}$, where $\mu>0$ is defined by \eqref{flatness}, $\varphi(0)=0$ and $0\in\partial\{u>\varphi\}$. Then there exists $\delta_0>0$ such that whenever $0\le\delta\leq\delta_0$ and
		\begin{equation*}\label{4.11}
			|\nabla\varphi(0)|\le\frac{\delta_0}{8^{\tau (k-1)}}
		\end{equation*}
		for some integer $k>0$, then
		$$
		\sup_{B_{1/8^k}}|u|\le\frac{1}{8^{(1+\tau)k}},
		$$
		where \begin{equation}\label{tau}
			\tau:=\min\left\{\beta,\frac{\gamma}{p-\gamma}\right\}.
		\end{equation}
	\end{lemma}
	\begin{proof}
		We argue inductively. The case of $k=1$ follows from Lemma \ref{l4.1}. Indeed, let $\varepsilon=\frac 1 8$ in Lemma \ref{l4.1} and choose
		\begin{equation}\label{choice}
			\delta_0:=\min\left\{\delta_\varepsilon,\frac \mu 2\right\},
		\end{equation}
		where the constants $\mu>0$ and $\delta_\varepsilon>0$ are as in the Lemma \ref{l4.1}. The latter insures that whenever $0\le\delta\le\delta_0$ and
		$$
		|\nabla\varphi(0)|\le\delta_0,
		$$
		then (recall that $\tau\le\beta$)
		$$
		\sup_{B_{1/8^k}}|u|\le\frac{1}{8^{1+\beta}}\le\frac{1}{8^{1+\tau}}.
		$$
		We now suppose that conclusion of the Lemma holds for $k=j>1$ and aim to conclude that it holds also for $k=j+1$. Thus, we assume  
		$$
		|\nabla\varphi(0)|\le \frac{\delta_0}{8^{\tau j}},
		$$
		and aim to conclude that
		\begin{equation}\label{induction}
			\sup_{B_{1/8^{j+1}}}|u|\le\frac{1}{8^{(1+\tau)(j+1)}}.
		\end{equation}	
		Observe that the definition of $\tau$ guarantees that
		$$
		\tilde{u}(x):= 8^{(1+\tau)j}u\left(\frac{x}{8^j}\right),\quad x\in B_1,
		$$
		is a minimizer of $I_{\tilde{\delta}}$ (for a $0\le\tilde{\delta}\le\delta_0$) with the obstacle 
		$$
		\tilde{\varphi}(x):= 8^{(1+\tau)j}\varphi\left(\frac{x}{8^j}\right),\quad x\in B_1.
		$$
		Indeed, set $\delta:=\tilde{\delta}8^{(\tau\gamma-p(1-\tau))j}$. Note that \eqref{tau} implies $\delta<\delta_0$. As $u$ is a minimizer of \eqref{4.1} and
		\begin{equation*}
			\begin{aligned}
				I_{\tilde{\delta}}(\tilde{u})&=\int_{B_1}\left(\frac{|\nabla \tilde{u}|^p}{p}+\tilde{\delta}(\tilde{u}-\tilde{\varphi})^\gamma\right)\,dx\\
				&=8^{-(n+(1-\tau)p)j}\int_{B_{1/8^j}}\left(\frac{|\nabla u|^p}{p}+\delta(u-\varphi)^\gamma\right)\,dx,
			\end{aligned}
		\end{equation*}
		then $\tilde{u}$ is a minimizer of $I_{\tilde{\delta}}$ over the functions that stay above $\tilde{\varphi}$.
		
		To apply the previous lemma for the pair $\tilde{u}$, $\tilde{\varphi}$, we make sure its assumptions are satisfied. By the introductory assumption
		$$
		\sup_{B_{1/8^{j}}}|u|\le\frac{1}{8^{(1+\tau)j}},
		$$
		therefore	
		$$
		\|\tilde{u}\|_{L^\infty(B_1)} \leq 1. 
		$$
		On the other hand, as $\varphi\in C^{1,\beta}(B_1)$ and $\varphi(0)=0$, recalling \eqref{L}, \eqref{tau}, \eqref{choice}, one has
		$$
		\sup_{B_{1/8^j}}\left|\varphi(x)-\nabla\varphi(0)\cdot x\right|\le\frac{\mu L}{2L\cdot8^{(1+\beta)j}}\le\frac{\mu}{2\cdot8^{(1+\tau)j}}.
		$$
		Hence,
		\begin{equation*}
			\begin{aligned}
				\|\tilde{\varphi}\|_{L^\infty(B_1)}&=8^{(1+\tau)j}\sup_{B_{1/8^j}}|\varphi|\\
				&\le8^{(1+\tau)j}\left[\frac{\mu}{2\cdot8^{(1+\tau)j}}+\frac{|\nabla\varphi(0)|}{8^j}\right]\\
				&\le\frac{\mu}{2}+\delta_0\\
				&\le\mu.
			\end{aligned}
		\end{equation*}	
		Also,  
		$$
		|\nabla\tilde{\varphi}(0)|\le\delta_0.
		$$
		Lemma \ref{l4.1} then implies, for $\varepsilon=\frac 1 8$,
		$$
		\sup_{B_{1/8}}|\tilde{u}|\le\frac{1}{8^{1+\tau}}.
		$$
		The latter gives \eqref{induction}. 
	\end{proof}

	We are now ready to prove the main result of this section.
	\begin{theorem}\label{t4.1}
		Let $\Omega\subset\mathbb{R}^n$ be a bounded domain, $2\leq p<+\infty$, $\varphi\in C^{1,\beta}(\Omega)$, for a $\beta\in(0,1]$ and $g\in W^{1,p}(\Omega)$. If $u$ is a solution of \eqref{1.1}, and $x_0\in \partial\{u>\varphi\}\cap\Omega$, then there exist positive universal constants $\kappa$, $C$ and $\rho_0$, depending only on $p$, $\gamma$, $\dist(x_0,\partial\Omega)$, $\|u\|_{L^{\infty}(\Omega)}$ and $\|\varphi\|_{C^{1,\beta}(\Omega)}$, such that whenever
		\begin{equation}\label{4.14}
			|\nabla\varphi(x_0)|\le\kappa\rho^{\tau}, 
		\end{equation}
		for $0<\rho<\rho_0$, then
		$$
		\sup_{B_{\rho}}|u(x)-u(x_0)-\nabla u(x_0)\cdot(x-x_0)|\le C \rho^{1+\tau},
		$$	
		where $\tau=\min\left\{\beta,\frac{\gamma}{p-\gamma}\right\}$. Consequently,
		$$
		\sup_{y\in B_{\rho}(x_0)}|u(y)-\varphi(x_0)| \leq C\rho^{1+\tau}.
		$$	
	\end{theorem}
	\begin{proof}
		Without loss of generality, we may assume that $x_0=0$, $\varphi(0)=0$ and $\Omega=B_1$. Let $\delta_0$ be as in Lemma \ref{l4.2} and set
		$$
		\tilde{u}(x):=\frac{u(x)}{M}\quad \mbox{and} \quad \tilde{\varphi}(x):=\frac{\varphi(x)}{M},
		$$
		for a constant
		$$
		M\geq\max \left\{\|u\|_{L^\infty(B_1)}, \delta_0^{\frac{1}{\gamma-p}}\right\}.
		$$
		Observe that $\tilde{u}$ is a minimizer for $I_{\delta_0}$ with obstacle $\tilde{\varphi}$. Also, 
		$\|\tilde{u}\|_{L^{\infty}(B_1)}\leq1$ and $\|\tilde{\varphi}\|_{L^\infty(B_1)}\le\delta_0$. Using \eqref{4.14}, for $\kappa>0$ small enough, one has
		\begin{equation}\label{4.15}
			|\nabla\tilde{\varphi}(0)|\leq\frac{\kappa}{M}r^\tau\leq\delta_0r^\tau.
		\end{equation}
		Now if $0<r\leq\frac 1 4$, we choose $k\in\mathbb{N}$ such that 
		$$
		2^{-(k+1)}<r\leq 2^{-k}.
		$$
		By Lemma \ref{l4.2}, \eqref{4.15} implies
		$$
		\sup_{B_{r}}|\tilde{u}|\le \sup_{B_{2^{-k}}}|\tilde{u}|\leq2^{-k (1+\tau)}\leq2^{(1+\tau)}r^{1+\tau}.
		$$
		Consequently, 
		$$
		\sup_{B_{r}}|u|\le Cr^{1+\tau},
		$$
		for a constant $C>0$ depending only on $p$, $\gamma$, $\|\varphi\|_{C^{1,\beta}(B_1)}$ and $\|u\|_{L^{\infty}(B_1)}$. Recalling Theorem \ref{existence}, note that the last dependence of $C$ on $\|u\|_{L^{\infty}(B_1)}$ can be replaced by $\|g\|_{W^{1,p}(B_1)}$.
	\end{proof}
	
	\section{Scaling adjustment}\label{s8}
	Our next goal is to obtain sharp regularity for minimizers of \eqref{1.1} at the free boundary points, where the gradient of the minimizer is large. The intuition behind the proof is that the problem should behave essentially as an obstacle problem, governed by a uniformly elliptic operator. Unlike the (classical) obstacle problem for $p$-Laplacian, \cite{ALS}, when the solution $u$ can be interpreted as the minimal solution of a uniformly elliptic equation with a boundded right hand side, in singular setting it is not clear whether it will solve the corresponding equation, as our functional is not convex in $u$. Moreover, even if one is able to conclude that $u$ indeed solves \eqref{EL}, as the right hand side in \eqref{EL} blows up at the free boundary points, we still would not be able to use the elliptic regularity theory.
	
	In this section, to circumvent these difficulties, we use a scaling argument. The idea is to scale the terms near the free boundary points so the problem looks like a linear elliptic equation. One technical difficulty in this approach is that when scaling (since the gradient in this case is bounded away from zero), the corresponding linear terms blow up, as the scaling goes to infinity. To avoid it, we subtract the linear part of the gradient in the functional. This adjusted scaling then ensures that in the limit, we arrive at a linear elliptic problem without the (zero order) singular term.
	
	To proceed, for $\xi$, $a\in\R^n$, we define
	$$
	\tilde{H}(\xi):=\frac{1}{p\varepsilon^2}\left(|\varepsilon\xi+a|^p-|a|^p-p\varepsilon|a|^{p-2} a\cdot\xi\right).
	$$
	\begin{lemma}\label{l5.1}
		Let $\varepsilon$, $\delta\in(0,1]$, $R>0$, $\varphi\in C^{1,\beta}(B_R)$, $g\in W^{1,p}(B_R)$, $\kappa_0>0$ and $a\in\R^n$ be such that $|a|>\kappa_0$. If  $u\in W^{1, p}(B_R)$ is a minimizer of 
		$$
		\int_{B_R}\left(\tilde{H}(\nabla v)+\delta (v-\varphi)^\gamma\right)\,dx,
		$$
		over the set $\mathbb{K}_R$, defined by \eqref{minimizationset}, such that
		$$
		|\varepsilon\nabla u+a|>\kappa_0,
		$$
		then there exist $C>0$ and $\alpha\in(0,1)$ constants, depending only on $\kappa_0$, $\|\varphi\|_{C^{1, \beta}(B_R)}$ and  $\|u\|_{L^{\infty}(B_R)}$, such that
		$$
		\|u\|_{C^{1,\alpha}(B_{R/2})}\le C.
		$$
	\end{lemma}
	\begin{proof}
		Observe that 
		\begin{equation*}\label{6.1}
			H_0(\xi)\ge\tilde{H}(\xi),
		\end{equation*}
		where the function $H_0$ is defined by \eqref{H} of Appendix \ref{A}. Also
		\begin{equation*}\label{6.2}
			H_0(\xi)=\tilde{H}(\xi),\,\,\,\textrm{whenever}\,\,\,|\varepsilon\xi+a|>\kappa_0.
		\end{equation*}
		Therefore, if $v\in\mathbb{K}_R$, then	
		\begin{align*}
			\int_{B_R}\left(H_0(\nabla u)+\delta (u-\varphi)^\gamma\right)\,dx&=	\int_{B_R}\left(\tilde{H}(\nabla u)+\delta (u-\varphi)^\gamma\right)\,dx\\
			&\le\int_{B_R}\left(\tilde{H}(\nabla v)+\delta(v-\varphi)^\gamma\right)\,dx\\
			&\le\int_{B_R}\left(H_0(\nabla v)+\delta(v-\varphi)^\gamma\right)\,dx,
		\end{align*}
		that is, $u$ is a minimizer of 	
		$$
		\int_{B_R}\left(H_0(\nabla v)+\delta(v-\varphi)^\gamma\right)\,dx
		$$
		over the set $\mathbb{K}_R$. Theorem \ref{t3.1} them implies the desired result.
	\end{proof}
	The next lemma is the main step towards our goal.
	\begin{lemma}\label{l5.2}
		If $u$ is a minimizer of \eqref{1.1} for $\Omega=B_1$, $0\in \partial\{u>\varphi\}$ and $0<\kappa\le|\nabla u|\le\Gamma$ for some constants $\kappa$ and $\Gamma$, then there exists $C>0$, depending only on $n$, $p$, $\kappa$, $\Gamma$, $\|u\|_{L^{\infty}(B_1)}$ and $\|\varphi\|_{C^{1,\beta}(B_1)}$, such that for every $r\in(0,1)$ either
		$$
		\|u-\varphi\|_{L^\infty(B_r)}\le Cr^{\theta},
		$$
		or there exists $j\in\mathbb{N}$ with $2^jr<1$ such that
		$$
		\|u-\varphi\|_{L^\infty(B_r)}\le 2^{-j\theta}\|u-\varphi \|_{L^\infty(B_{2^jr})}, 
		$$
		where
		$$
		\theta:=\min\left\{1+\beta,\, \frac{2}{2-\gamma}\right\}.
		$$
	\end{lemma}
	\begin{proof}
		We argue by contradiction and assume that there exists sequence of minimizers $u_k$, obstacles $\varphi_k$ and radii $r_k$ with
		$$
		\max\left\{\|u_k\|_{L^{\infty}(B_1)},\|\varphi_k\|_{L^\infty(B_1)}\right\}\leq M<+\infty,
		$$
		such that
		\begin{equation}\label{6.3}
			c_k:=\|u_k-\varphi_k \|_{L^\infty(B_{r_k})}>k r_k^{\theta}.
		\end{equation}
		and $\forall j\in\mathbb{N}$ with $2^jr_k<1$,
		\begin{equation}\label{6.4}
			\|u_k-\varphi_k \|_{L^\infty(B_{r_k})}>2^{-j\theta}\|u_k-\varphi_k\|_{L^\infty(B_{2^jr_k})}.
		\end{equation}
		Observe that the boundedness of $u_k$, $\varphi_k$ combined with \eqref{6.3} implies 
		\begin{equation}\label{6.5}
			r_k\to0.
		\end{equation}
		On the other hand, by Theorem \ref{t3.1},
		\begin{equation}\label{6.6}
			c_k\le C r_k^{1+\alpha},
		\end{equation}
		with 
		$$
		\alpha\le\min\left\{\beta,\frac{\gamma}{2-\gamma}\right\}.
		$$
		Thus, $1+\alpha\le\theta$.
		As $0\in\partial\{u_k>\varphi_k\}$, then $\nabla u_k(0)=\nabla\varphi_k(0):=a_k$. The strategy now is to apply Lemma \ref{l5.1} to suitably scaled functions. Set
		$$
		\tilde{u}_k(x):=\frac{1}{c_k}\left[u_k(r_kx)-u_k(0)-a_k\cdot(r_kx)\right]
		$$
		and
		$$
		\tilde{\varphi}_k(x):=\frac{1}{c_k}\left[\varphi_k(r_kx)-\varphi_k(0)-a_k\cdot(r_kx)\right].
		$$
		Note that 
		\begin{equation}\label{obstacle}
			\tilde{u}_k\ge\tilde{\varphi}_k\,\,\textrm{ in }\,\,B_r,\,\, \forall r\le\frac 1 r_k.
		\end{equation}
		Also $\tilde{u}_k(0)=\tilde{\varphi}_k(0)=0$, $\nabla\tilde{u}_k(0)=\nabla \tilde{\varphi}_k(0)=0$ with 
		\begin{equation}\label{unitnorm}
			\|\tilde{u}_k-\tilde{\varphi}_k\|_{L^\infty(B_{1})}=1,
		\end{equation}
		and, recalling \eqref{6.4},
		$$
		\|\tilde{u}_k-\tilde{\varphi}_k \|_{L^\infty(B_{2^j})}<2^{j\theta}.
		$$
		Furthermore,
		\begin{equation}\label{6.7}
			\|\tilde{\varphi}_k\|_{C^{1,\beta}(B_{1/r_k})}\le\frac{Cr_k^{1+\beta}}{c_k}  \|{\varphi}_k\|_{C^{1,\beta}(B_1)}\le\frac{C}{k}  \|{\varphi}_k\|_{C^{1,\beta}(B_1)}\to0.
		\end{equation}
		In the second estimate we used the fact that $1+\alpha\le\theta$ (in fact, the inequality is strict, otherwise \eqref{6.3} contradicts to \eqref{6.6}) together with \eqref{6.3}, \eqref{6.5} and \eqref{6.6}. Note that
		$$
		u_k(x)=c_k\tilde{u}_k\left(\frac{x}{r_k}\right)+u_k(0)+a_k\cdot x,\,\,\,x\in B_{1/r_k},
		$$
		and since $u_k$ is a solution of \eqref{1.1} with obstacle $\varphi_k$, then for any $r\leq \frac 1 {r_k}$, the function $\tilde{u}_k$ is a minimizer of	
		$$
		\int_{B_r}\left(\frac{|\frac{c_k}{r_k} \nabla v+a_k|^p}{p}+c_k^{\gamma}(v-\tilde{\varphi}_k)^\gamma\right)\,dx
		$$
		over the set of functions $v\in W^{1,p}(B_r)$ that stay above $\tilde{\varphi}_k$, i.e., $v\ge\tilde{\varphi}_k$. On the other hand, as $k\to\infty$, using \eqref{6.5}, \eqref{6.6} and \eqref{6.3}, one has 
		$$
		\varepsilon_k:=\frac{c_k}{r_k}\leq C r_k^{\alpha}\to0\,\,\,\text{ and }\,\,\,\delta_k:= \frac{c_k^\gamma}{\varepsilon_k^2}= \frac{r_k^2}{c_k^{2-\gamma}}\leq \frac{1}{k^{2-\gamma}}\to 0.
		$$
		Since
		$$
		\int_{B_r}\left(|a_k|^p+p\varepsilon_k |a_k|^{p-2}a_k\cdot\nabla v\right)\,dx
		$$
		depends only on boundary values of $v$, then $\tilde{u}_k$ is also a minimizer of 
		\begin{equation}\label{6.8}
			\int_{B_r}\left(H_k(\nabla v)+\delta_k (v-\tilde{\varphi}_k)^\gamma\right)\,dx,
		\end{equation}
		over the set of functions that stay above $\varphi_k$. Here
		$$
		H_k(\xi):=\frac{1}{p\varepsilon_k^2}\left(|\varepsilon_k\xi+a_k|^p-|a_k|^p-p\varepsilon_k |a_k|^{p-2}a_k\cdot \xi\right),\,\,\,\xi\in\R^n.
		$$
		As $\kappa\le|a_k|\le\Gamma$, then up to a subsequence $a_k\to a_\infty\in\R^n$ and $\kappa\le |a_\infty|\le\Gamma$. Consequently, as $k\rightarrow\infty$,
		\begin{equation}\label{6.9}
			H_k(\xi)\rightarrow\langle A,\xi\otimes\xi\rangle=\sum_{i,j=1}^n A_{ij}\xi_i\xi_j,
		\end{equation}
		locally uniformly in $\R^n$, where $A$ is the $n\times n$ strictly positive definite matrix defined below
		$$
		A:= |a_\infty|^{p-2}\mathbb{I}_n+(p-2)|a_\infty|^{p-4} a_\infty\otimes a_\infty \ge\kappa^{p-2}\mathbb{I}_n,
		$$
		with $\mathbb{I}_n$ being the $n\times n$ identity matrix. By Theorem \ref{t3.1}, $u_k$ is uniformly $C_{\loc}^{1,\alpha}$. Hence, for any $R>0$, one has
		$$
		\varepsilon_k|\nabla\tilde{u}_k(x)|= |\nabla u_k(r_k x)-\nabla u_k(0)|\leq C( r_kR)^\alpha,\,\,\,\forall x\in B_R.
		$$
		Thus, for fixed $R>0$ and large $k$, we can assume  that 
		$$
		\frac{\kappa}{2}\le|\varepsilon_k\nabla \tilde{u}_k+a_k|\leq2\Gamma.
		$$
		Lemma \ref{l5.1} then implies that $\tilde{u}_k$ is uniformly bounded in $C^{1,\alpha}(B_R)$. By Arzella-Ascoli theorem, there exists a function $u_*\in C^{1,\alpha}(\R^n)$, such that up to a subsequence, 
		\begin{equation}\label{6.10}
			\tilde{u}_k\to u_*\,\,\,\textrm{ and }\,\,\,\nabla\tilde{u_k}\to\nabla u_*
		\end{equation}
		locally uniformly in $\R^n$. Additionally, using \eqref{unitnorm}, \eqref{6.7} and \eqref{6.10}, one has
		\begin{equation}\label{norm}
			\|u_*\|_{L^\infty(B_1)}=1.
		\end{equation}
		Also, \eqref{obstacle}, \eqref{6.7} and \eqref{6.10} provide $u_*\ge0$.
		
		Next, we show that $u_*$ is $A$-harmonic, which implies (Liouville's theorem) that it has to be identically zero (as it vanishes at the origin). More precisely, if $U:=\{u_*>0\}\neq\emptyset$, then for any $B_r\subset U$ and $\psi\in C^{\infty}_0(B_r)$, we choose $\varepsilon_0>0$ small enough so that for  $\varepsilon<\varepsilon_0$ one has
		\[
		\inf_{B_r}\{u_*+\varepsilon \psi\}>0\text{ on } B_r.
		\]
		Recalling \eqref{6.7} and \eqref{6.10}, by uniform  convergence $\tilde{u}_k\to u_{*}$ and $\tilde{\varphi}_k\to 0$, for $k>k_\varepsilon$, one has
		\[
		\inf _{B_r} \{u_* +\varepsilon \psi \} - u_*> \varphi_k - \tilde{u}_k,
		\]
		and hence, 
		\[
		{\tilde{u}_k+\varepsilon \psi } \geq \varphi_k \text{ on } B_r.
		\]
		Since $\tilde{u}_k$ is a minimizer of \eqref{6.8} over the set of functions that stay above $\varphi_k$, then
		\begin{equation*}
			\begin{aligned}
				&\int_{B_r}\left(H_k(\nabla \tilde{u}_k)+\delta_k (\tilde{u}_k-\tilde{\varphi}_k)^\gamma\right)\,dx\\
				&\le\int_{B_r}\left(H_k(\nabla( \tilde{u}_k+\varepsilon \psi))+\delta_k (\tilde{u}_k+\varepsilon \psi-\tilde{\varphi}_k)^\gamma\right)\,dx.
			\end{aligned}
		\end{equation*}
		Using \eqref{6.9} and passing to the limit in the last inequality, we obtain
		\[
		I_A(u_*)\le I_A(u_*+\varepsilon\psi),
		\]
		where
		$$
		I_A(v):=\int_{B_r}\langle A,\nabla v\otimes \nabla v\rangle\,dx.
		$$
		Thus, $u_*$ is a minimizer of $I_A$ over the set of non-negative functions. Therefore, $u_*$ is $A$-harmonic in $\{u_*>0\}$, i.e.,
		\[
		\Div(A \nabla u_*)=0 \text{ in } \{u_*>0\}.
		\]
		Using similar reasoning with $\psi\geq 0$, we conclude  that $u_*$  is A-superharmonic on the whole space $\mathbb{R}^n$, i.e.,
		\[
		\Div(A \nabla u_*) \leq 0 \text{ in } \mathbb{R}^n.
		\]
		Now, if $B_R$ is any ball and $h^*_R$ is the $A$-harmonic function in $B_R$ that agrees with $u_*$ on $\partial B_R$, then by the maximum principle (see, for example, \cite[page 111]{HKM06}), $u_*\geq h^*_R\geq 0$. If $U_R:=\{u_*>h^*_R\}\neq\emptyset$, then since $u_*>0$ in $U_R$, by the previous argument, $u_*$ is $A$-harmonic in $U_R$, and $u_*=h^*_R$ on $\partial U_R$. Once again, using the maximum principle, we conclude that $u_*=h^*_R$ in $U_R$, which is a contradiction, and hence, $u_*=h_R$ in $B_R$. Since $B_R$ was arbitrary, then $u_*$ is everywhere $A$-harmonic. As it is also bounded, Liouville theorem (see, for example, \cite[page 112]{HKM06}) then insures that $u_*\equiv const$. Moreover, as by construction $\tilde{u}_k(0)=0$, then from \eqref{6.10} one has $u_*(0)=0$. Thus, $u_*\equiv0$, which contradicts to \eqref{norm}.
	\end{proof}
	\begin{corollary}\label{c5.1}
		If $u$ is a minimizer of \eqref{1.1} for $\Omega=B_1$, $0\in \partial\{u>\varphi\}$ and $0<\kappa\le|\nabla u|\le\Gamma$ for some constants $\kappa$ and $\Gamma$, then there exists $C>0$, depending only on $n$, $p$, $\kappa$, $\Gamma$, $\|u\|_{L^{\infty}(B_1)}$ and $\|\varphi\|_{C^{1,\beta}(B_1)}$, such that for every $r\in(0,\frac 12)$ one has
		$$
		\|u-\varphi\|_{L^\infty(B_r)}\le Cr^{1+\tau},
		$$
		where $\tau>0$ is defined by \eqref{tau}. 	
	\end{corollary}
	
	\section{Large gradient estimates}\label{s9}
	We are finally ready to prove sharp regularity of minimizers with large gradient at the free boundary.
	\begin{theorem}\label{t5.1}
		If $u$ is a minimizer of \eqref{1.1}, $x_0\in \partial\{u>\varphi\}\cap\Omega$ and $\kappa>0$ is as in Theorem \ref{t4.1}, then there exist positive constants $C$ and $\rho_0$, depending only on $p$, $\gamma$, $\dist(x_0,\partial\Omega)$, $\|u\|_{L^{\infty}(\Omega)}$ and $\|\varphi\|_{C^{1,\beta}(\Omega)}$, such that whenever
		$$
		|\nabla u(x_0)|> \kappa\rho^{\tau}, 
		$$
		for some $\rho<\rho_0$, then
		$$
		\sup_{B_{\rho}}|u(x)-u(x_0)-\nabla u(x_0)\cdot(x-x_0)|\le C \rho^{1+\tau},
		$$
		where $\tau=\min\left\{\beta,\frac{\gamma}{p-\gamma}\right\}$. 
	\end{theorem}
	\begin{proof}
		The idea is to apply Corollary \ref{c5.1} to a suitably rescaled function. For that purpose, set 
		$$
		\rho_*:=\left(\frac{|\nabla u(x_0)|}{\kappa}\right)^{\frac{1}{\tau}}
		$$
		and define
		$$
		\tilde{u}(x):=\frac{u(x_0+\rho_* x)-u(x_0)}{\rho_*^{1+\tau}},\quad \tilde{\varphi}(x):=\frac{\varphi(x_0+\rho_*x)-\varphi(x_0)}{\rho_*^{1+\tau}},
		$$
		Observe that $\tilde{u}$ is a minimizer of \eqref{1.1} for $\Omega=B_1$ with obstacle $\tilde{\varphi}$, and $0\in\partial\{\tilde{u}>\tilde{\varphi}\}$. For $\rho_*\le\rho_0$, Theorem \ref{t4.1} provides
		$$
		\sup_{x\in B_{\rho_*}(x_0)}|u(x)-u(x_0)| \le C \rho_*^{1+\tau}.
		$$
		Hence, $\tilde{u}$ is uniformly bounded in $B_1$, i.e., there exists a constant $C_1>0$, such that
		$$
		\|\tilde{u}\|_{L^{\infty}(B_1)}<C_1.
		$$
		As zero is a contact point, then
		$$
		|\nabla\tilde{\varphi}(0)|=|\nabla\tilde{u}(0)| = \frac{1}{\rho_*^\tau} |\nabla u (x_0)|=\kappa,
		$$
		and so there exists a universal constant $C_2>0$, such that
		$$
		\|\tilde{\varphi}\|_{C^{1, \beta}(B_1)}\le C\|\varphi\|_{C^{1, \beta}(\Omega)}<C_2.
		$$
		Note also that there exists $r_*>0$, depending only on $\kappa$, $C_1$ and $C_2$, such that 
		$$
		\frac{\kappa}{2}<|\nabla\tilde{u}(x)|<\frac{2}{\kappa}\,\,\,\textrm{in}\,\,\,B_{r_*}.
		$$
		Applying Corollary \ref{c5.1} in $B_{r_*}$ for the minimizer $\tilde{u}$ with obstacle $\tilde{\varphi}$, for $r\in(0,\frac{r_*}{2})$, we have
		$$
		\sup_{|x|<r}|\tilde{u}(x)-u(0)-\nabla \tilde{u}(0)\cdot x|\le C r^{1+\tau}.
		$$
		The latter, in terms of the function $u$, is
		$$
		\sup_{|x-x_0|<r}|u(x)-u(x_0)-\nabla u(x_0)\cdot(x-x_0)|\le C r^{1+\tau},
		$$
		where $0<r<\frac{\rho_*r_*}{2}$.
		
		If $\frac{\rho_*r_*}{2}\le r\le\rho_*$, then, as $B_r(x_0)\subseteq B_{\rho_*}(x_0)$, applying Theorem \ref{t4.1} with radius $\rho_*$, we estimate
		$$
		\sup_{|x-x_0|<r}|u(x)-u(x_0)-\nabla u(x_0)\cdot(x-x_0)|\le C\rho_*^{1+\tau}\le C {\frac{2}{r_*}}^{1+\tau}r^{1+\tau}.
		$$
		Finally, if $\rho_*>\rho_0$, then 
		$$
		|\nabla u(x_0)|>\kappa\rho_0^\tau,
		$$
		and Corollary \ref{c5.1} gives the desired estimate.	
	\end{proof}
	As a consequence of Theorem \ref{t4.1} and Theorem \ref{t5.1}, we obtain the following result.
	\begin{theorem}\label{mainresult}
		Let $\varphi\in C^{1,\beta}(\Omega)$, $g\in W^{1,p}(\Omega)$, $p\in[2,\infty)$ and $\gamma\in(0,1)$. If $u$ is a minimizer of \eqref{1.1}, then for any $\Omega'\subset\subset\Omega$, there exist universal constants $C>0$ and $r_0>0$ such that for any $r\in(0,r_0)$ and for any $y\in\partial\{u>\varphi\}\cap\Omega'$, one has
		$$
		\sup_{x\in B_r(y)}|u(x)-u(y)-\nabla u(y)\cdot(x-y)|<Cr^{1+\tau},
		$$
		where $\tau=\min\left\{\beta,\frac{\gamma}{p-\gamma}\right\}$, i.e., $u$ is $C_{\loc}^{1,\tau}$ at $\partial\{u>\varphi\}$, and this regularity is optimal.
	\end{theorem}
	
	\medskip
	
	\textbf{Acknowledgments.} DJA is partially supported by CNPq 311138/2019-5 and grant 2019/0014 Paraiba State Research Foundation (FAPESQ). RT and VV are partially supported by FCT - Funda\c{c}\~ao para a Ci\^encia e a Tecnologia, I.P., through project PTDC/MAT-PUR/28686/2017 and by the Centre for Mathematics of the University of Coimbra - UIDB/00324/2020, funded by the Portuguese Government through FCT/MCTES.
	
	\appendix
	\section{}\label{A}
	In this appendix we prove some properties of the auxiliary functions $H$ and $G$, that are used in Section \ref{s5}. We also construct a function $H_0$ satisfiying \eqref{assumptions}, which is used in the proof of Lemma \ref{l5.1}.
	\begin{lemma}\label{Hproperty}
		If $H$ satisfies \eqref{assumptions} and $G$ is defined by \eqref{3.1}, then
		\begin{equation}\label{leq:HPQ}
			H(x)-H(y)-\nabla H(y)\cdot(x-y)\geq c\,G(|x-y|),
		\end{equation}
		where $c>0$ is a constant depending only on $\kappa_1$, $\kappa_2$, $\Lambda$, $\lambda$, $\Upsilon$ and $p$.
	\end{lemma}
	\begin{proof}
		Define
		$$
		\tilde{\omega}(z):=\frac{\omega(z)}{z}=\kappa_1z^{p-2}+\kappa_2
		$$
		and observe that
		$$
		G(|x-y|)\leq C\left(\tilde{\omega}(|x|)+\tilde{\omega}(|y|)\right)|x-y|^2
		$$
		and
		$$
		\int_0^1(1-s)\tilde{\omega}(|(1-s)x + sy|)\,ds\ge c\,\left(\tilde{\omega}(|x|)+\tilde{\omega}(|y|)\right),
		$$
		for any $x$, $y\in\R^n$. Then
		\begin{equation*}
			\begin{aligned}
				&H(x)-H(y)-\nabla H(y)\cdot(x-y)\\
				&=\int_0^1(1-s)\langle D^2 H\left((1-s)x+sy\right),(x-y)\otimes(x-y)\rangle\,ds\\
				&\geq c|x-y|^2\int_0^1(1-s)\,\tilde{\omega}(|(1-s)x+sy|)\,ds\\
				&\geq c\left(\tilde{\omega}(|x|)+\tilde{\omega}(|y|) \right)|x-y|^2
				\geq c\,G(|x-y|).
			\end{aligned}
		\end{equation*}
	\end{proof}
	\begin{lemma}\label{l3.3}	
		If $u,v\in W^{1,p}(B_R)$, then there exists a constant $C>0$, depending only on $p$ and $n$, such that for $p\geq 1$, we have
		\begin{equation*}
			\begin{aligned}
				\int_{B_R}G(|\nabla u- (\nabla u)_{R})|)\,dx&\le 
				C\int_{B_R}G(|\nabla v-(\nabla v)_{R})|)\,dx\\
				&+C
				\int_{B_R}G( |\nabla u-\nabla v|)\,dx.
			\end{aligned}
		\end{equation*}
	\end{lemma}
	\begin{proof}
		Note that
		$$
		G(x+y)\leq C(G(x)+G(y)).
		$$ 
		Observe also
		\begin{align}\label{lemma3.3proof}
			G(|\nabla u-(\nabla u)_{R}|) &\le C\left(G(|\nabla v-(\nabla u)_{R}|)+G(|\nabla u-\nabla v|)\right)\nonumber \\
			&\leq C \Big(G(|\nabla v-(\nabla v)_{R}|)+   G(|(\nabla u)_{R}-(\nabla v)_{R}|)\nonumber \\
			&\quad\,+G(|\nabla u-\nabla v|)\Big).
		\end{align}	
		Using the convexity of $G$ and Jensen's inequality, one has
		$$
		G(|(\nabla u)_{R}-(\nabla v)_{R}|)=G(|\nabla(u-v)_{R}|) 
		\leq \Big(G(|\nabla(u-v)|)\Big)_{R}.
		$$
		Integrating the latter over $B_R$, we get
		\begin{align*}
			\int_{B_R}G(|(\nabla(u-v)_{R}|)\,dx&\le\int_{B_R} \Big(G(|\nabla (u-v)|)\Big)_{R}\,dx\\
			&=\int_{B_R}G(|\nabla(u-v)|)\,dx,
		\end{align*}
		therefore, integrating \eqref{lemma3.3proof} over $B_R$, we get the desired estimate.
	\end{proof}
	Note that if $v\in W^{1,p}(B_R)$, $r>0$ is a minimizer of the functional
	$$
	\int_{B_R}H(\nabla v)\,dx,
	$$
	then in $B_R$ it solves the Euler-Lagrange equation
	$$
	\Div\left(\nabla H(\nabla v)\right)=0,
	$$
	i.e., $v$ is $H$-harmonic (and hence, has H\"older continuous gradient, \cite{HKM06}). Therefore (see \cite[Lemmas 5.1 and  5.2]{L91} for the proof), the following lemma holds.
	\begin{lemma}\label{l3.2}
		If $v\in W^{1,p}\left(B_R)\right)$ is a minimizer of
		$$
		\int_{B_R}H(\nabla v)\,dx,
		$$
		then there exist $C>0$ and $\sigma>0$ constants depending only on $\kappa_1, \kappa_2$ and $\|v\|_{L^{\infty}(B_r)}$, such that for each $0<r<R$, one has
		\begin{align*}\label{ineq:rCR}
			\int_{B_r} G(|\nabla v-(\nabla v)_{r}|)\,dx\leq C \left(\frac{r}{R}\right)^{n+q\sigma}\int_{B_R} G(|\nabla v-(\nabla v)_{R}|)\,dx,
		\end{align*}
		where $q$ is defined by \eqref{q}.	
	\end{lemma}
	\subsection*{An auxiliary function}
	We close this appendix by constructing a function $H_0$ satisfying \eqref{assumptions}. Its role is evident in the section \ref{s8}, when elaborating a scaling argument aimed to obtaining sharp regularity of minimizers of \eqref{1.1} with large gradient at the free boundary points. To construct it, first we choose $\nu>0$ small enough so that the function 
	$$
	h(z):=|z|^p+\nu |z|^2(\kappa_0^2-|z|^2)^3\chi_{\{|z|\le\kappa_0\}},\quad z\in\R^n
	$$
	is of class $C^2(\R^n)$ with
	$$
	|D^2h(z)|\ge\nu_1 |z|^{p-2}+\nu_2,
	$$
	for some positive constants $\nu_1$ and $\nu_2$, depending only on $\kappa_0$ and $p$. Here $\chi_E$ is the characteristic function of the set $E$.
	\begin{figure}[h]
		\begin{center}
			\scalebox{0.7} 
			{
				\begin{pspicture}(0,-6.305475)(12.114815,2.694525)
					\definecolor{colour6}{rgb}{0.5019608,0.7019608,0.7019608}
					\definecolor{colour7}{rgb}{0.5019608,0.6,1.0}
					\definecolor{colour8}{rgb}{1.0,0.4,0.4}
					\psframe[linecolor=black, linewidth=0.04, dimen=outer](12.114815,2.694525)(0.0,-6.305475)
					\psbezier[linecolor=colour6, linewidth=0.04](0.8962963,2.1611917)(0.91481483,1.0130435)(1.8962963,-3.1388085)(2.6962962,-3.738808373591023)(3.4962964,-4.3388085)(3.1148148,-3.9869566)(4.014815,-4.6869564)(4.914815,-5.3869567)(4.9152,-5.8591895)(5.914815,-5.8869567)(6.914429,-5.9147234)(7.036554,-5.4)(7.836554,-4.8)(8.636554,-4.2)(8.696297,-4.238808)(9.296296,-3.6388085)(9.8962965,-3.0388083)(10.996296,0.86119163)(10.996296,2.0611916)
					\rput[bl](3.9851851,1.4908212){\small{$h(z),\textrm{ with }\kappa_0=1,\, \nu=1$, is convex} }
					\psline[linecolor=colour7, linewidth=0.04](2.9629629,1.5908213)(3.762963,1.5908213)
					\psline[linecolor=colour6, linewidth=0.04](2.9629629,1.0908213)(3.7407408,1.094525)
					\psline[linecolor=colour8, linewidth=0.04, linestyle=dashed, dash=0.17638889cm 0.10583334cm](3.0148149,0.5760064)(3.8148148,0.5760064)
					\rput[bl](10.092592,-3.5945492){\large{$h(z)$}}
					\rput[bl](3.9629629,0.92045087){\small{$h(z),\textrm{ with }\kappa_0=1,\, \nu=1.5$, is not convex} }
					\rput[bl](3.988889,0.4026731){\small{$|z|^3$} }
					\psbezier[linecolor=colour7, linewidth=0.04](0.93655396,2.2)(1.036554,0.2)(1.8365539,-2.6)(2.436554,-3.4999999999999987)
					\psbezier[linecolor=colour7, linewidth=0.04](2.436554,-3.5)(2.5365539,-3.8)(4.636554,-6.0)(6.036554,-5.9)(7.436554,-5.8)(8.536554,-4.6)(9.336554,-3.6)
					\psbezier[linecolor=colour7, linewidth=0.04](9.336554,-3.6)(10.336554,-2.0)(10.936554,0.9)(10.978221,2.0)
					\psbezier[linecolor=colour8, linewidth=0.03, linestyle=dashed, dash=0.17638889cm 0.10583334cm](0.8962963,2.1611917)(1.0148149,1.3079389)(1.636554,-5.8)(5.8962965,-5.88695652173913)(10.156038,-5.973913)(10.8962965,1.3563768)(10.996296,2.0605898)
				\end{pspicture}
			}
			\caption{The function $h(z)$ for $p=3$ and certain values of $\kappa_0$ and $\nu$.}
		\end{center}
	\end{figure}
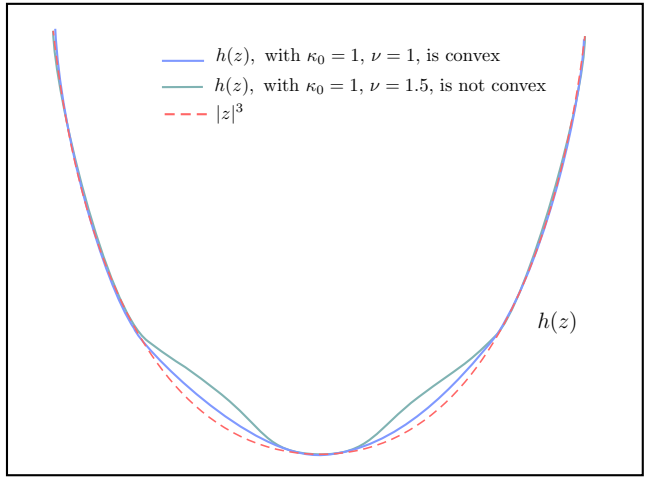
	Observe that for fixed $p\ge2$, $\kappa_0>0$ and small $\nu>0$, the function $h(z)$ is strictly convex (see below), agrees with $|z|^p$ away from zero, and its graph is never below $|z|^p$ (see Figure 2). We now set
	\begin{equation}\label{H}
		H_0(\xi):=\frac{1}{p\varepsilon^2}\left(h(\varepsilon\xi+a)-h(a)-\varepsilon\nabla h(a)\cdot\xi\right).
	\end{equation}
	Below we check that $H_0$ satisfies the conditions \eqref{assumptions}. Direct computation reveals
	\begin{equation*}
		\begin{aligned}
			\nabla h(z)&=p|z|^{p-2} z +2\nu \left[(\kappa_0^2-|z|^2)^3 - 3(\kappa_0^2-|z|^2)^2|z|^2\right] z\chi_{\{|z|\leq \kappa_0\}}\\
			&=p|z|^{p-2} z +2\nu(\kappa_0^2-4|z|^2) (\kappa_0^2-|z|^2)^2z\chi_{\{|z|\leq \kappa_0\}},
		\end{aligned}
	\end{equation*}
	and
	\begin{equation}\label{hessian}
		\begin{aligned}
			D^2 h(z)&=p|z|^{p-2}\mathbb{I}_n+   p(p-2)|z|^{p-4} z\otimes z\\
			+
			&2\nu(\kappa_0^2-4|z|^2) (\kappa_0^2-|z|^2)^2 \mathbb{I}_n\chi_{\{|z|\leq \kappa_0\}}\\
			-&12\nu(\kappa_0^4-3\kappa_0^2 |z|^2+2|z|^4)z\otimes z \chi_{\{|z|\leq\kappa_0\}},
		\end{aligned}
	\end{equation}
	where $\mathbb{I}_n$ is the $n\times n$ identity matrix. Note that the last two terms in \eqref{hessian} are bounded, hence	
	\begin{equation}\label{eq:d2hub}
		|D^2 h(z)|\leq p(p-1)|z|^{p-2}+C_{p, \kappa_0,\nu},
	\end{equation}
	for a constant $C_{p,\kappa_0,\nu}>0$ depending only on $p$, $\kappa_0$ and $\nu$.
	
	To have a lower bound on the quadratic form $\eta^TD^2 h(z)\eta$, $\eta\in\R^n$, we will consider two cases $|z|<\kappa_0/4$ and $|z|\geq\kappa_0/4$.
	In the first case, we have
	\begin{equation*}
		\begin{aligned}
			\eta^TD^2h(z)\eta&=p|z|^{p-2}|\eta|^2 +p(p-2)|z|^{p-4}|z\cdot \eta|^2+\nu(\kappa_0^6 |\eta|^2-10\kappa_0^4|z\cdot\eta|^2)\\
			&\geq\left(p|z|^{p-2}+6\nu \kappa_0^6\right)|\eta|^2.
		\end{aligned}
	\end{equation*}
	In the second case, note that
	\begin{equation*}
		\begin{aligned}
			\eta^TD^2h(z)\eta&=p|z|^{p-2}|\eta|^2 +p(p-2)|z|^{p-4}|z\cdot\eta|^2\\
			&\quad+
			2\nu(\kappa_0^2-4|z|^2) (\kappa_0^2-|z|^2)^2|\eta|^2\chi_{\{|z|\leq\kappa_0\}}\\
			&\quad-12\nu (\kappa_0^4-3\kappa_0^2 |z|^2+2|z|^4)|z\cdot\eta|^2 \chi_{\{|z|\leq\kappa_0\}}\\
			&\geq p|z|^{p-2}|\eta|^2-2\nu\kappa_0^6 |\eta|^2-12\nu(\kappa_0^4+2|z|^4)|z|^2 |\eta|^2\chi_{\{|z|\leq\kappa_0\}}\\
			&\geq\left(p|z|^{p-2}-38\nu\kappa_0^6\right)|\eta|^2\\
			&\geq\left(\frac{p}{2}|z|^{p-2}+\frac{p}{2}|\kappa_0/4|^{p-2}-38\nu\kappa_0^6\right)|\eta|^2\\
			&\geq\left(\frac{p}{2}|z|^{p-2}+\frac{p}{4}|\kappa_0/4|^{p-2}\right)|\eta|^2,
		\end{aligned}
	\end{equation*}
	where the last inequality holds for $\nu\leq\frac{p}{152}|\kappa_0/4|^{p-8}$. To sum up, for $\nu>0$ small enough, we have
	\begin{align}\label{eq:d2hlb}
		\eta^TD^2h(z)\eta
		\geq\left(c_1|z|^{p-2}+c_2\right) |\eta|^2,
	\end{align}	
	where $c_1>0$ is a constant depending only on $p$ and $c_2>0$ depends only on $p$ and $\kappa_0$.
	
	We are now ready to see that $H_0$ defined by \eqref{H} satisfies \eqref{assumptions}. We have
	$$
	\nabla H_0(\xi) = \frac{1}{p\varepsilon}\left(\nabla h(\varepsilon\xi+a)-\nabla h(a)\right)=
	\frac{1}{p}\int_0^1D^2 h(t\varepsilon\xi+a)\cdot\xi\,dt,
	$$
	therefore, from \eqref{eq:d2hub},
	$$
	|\nabla H_0(\xi)|\leq C_p(\varepsilon^{p-2} |\xi|^{p-2} +|a|^{p-2} +1)|\xi|,
	$$
	for a constant $C_p>0$ depending only on $p$. Thus, the first inequality of \eqref{assumptions} is satisfied. Also
	$$
	D^2H_0(\xi)=\frac{1}{p}D^2h(\varepsilon\xi+a),
	$$
	and similarly
	$$
	|D^2H_0(\xi)|\leq C_p(\varepsilon^{p-2} |\xi|^{p-2} +|a|^{p-2} +1).
	$$
	Thus, $H_0$ satisfies the second inequality of \eqref{assumptions} as well. To check the last inequality of \eqref{assumptions}, observe that from \eqref{eq:d2hlb} we have
	\begin{align*}
		\eta^TD^2H_0(\xi)\eta
		\geq\left(c_1|\varepsilon\xi+a|^{p-2}+c_2\right)|\eta|^2.
	\end{align*}
	On the other hand, for $|\xi|>2|a|/\varepsilon$, we estimate
	$$
	c_1|\varepsilon\xi+a|^{p-2}+c_2\geq c_1\left(\varepsilon|\xi|-|a|\right)^{p-2}+c_2\ge2^{2-p} c_1\varepsilon^{p-2}|\xi|^{p-2}+c_2,
	$$
	and for
	$|\xi|\leq 2|a|/\varepsilon$, one has
	\begin{equation*}
		\begin{aligned}
			(\varepsilon^{p-2}|\xi|^{p-2}+|a|^{p-2} +1)&\leq2^{p-2}|a|^{p-2}+|a|^{p-2}+1\\
			&\leq\frac{1}{c_2}\left(2^{p-2}|a|^{p-2}+|a|^{p-2}+1\right)\left(c_1|\varepsilon\xi+a|^{p-2}+c_2\right).
		\end{aligned}
	\end{equation*}
	Thus, there exist $\tilde{c}_1, \tilde{c}_2>0$ depending uniformly on $|a|$, $p$, and $\kappa_0$, such that
	\begin{align*}
		\eta^TD^2H_0(\xi)\eta
		\geq(\varepsilon^{p-2} \tilde{c}_1|\xi|^{p-2}+\tilde{c}_2) |\eta|^2.
	\end{align*}	
	It remains to take $\kappa_1 = \varepsilon^{p-2} \tilde{c}_1, \kappa_2=\tilde{c}_2$ and note that $\kappa_1, \kappa_2$ are uniformly bounded, while $\kappa_1+\kappa_2$ is uniformly away from zero. Thus, $H_0$ satisfies all three inequalities of \eqref{assumptions}.
	
	\section{}\label{B}
	For the reader's convenience, we collect here some known results that were used in the paper. The first result is from \cite[Lemma 5.13]{GM12}.
	\begin{lemma}\label{l2.1}
		If $\phi(\rho)\ge0$ is a non-decreasing function and
		$$
		\phi(\rho)\le A\left[\left(\frac{\rho}{R}\right)^\alpha+\varepsilon\right]\phi(R)+BR^\beta,
		$$
		for some $A$, $\alpha$, $\beta>0$, with $\alpha>\beta$ and for all $0\le\rho\le R\le R_0$, where $R_0>0$ is given, then there exist positive constants $\varepsilon_0$ and $c$, depending only on $A$, $\alpha$, $\beta$, such that if $\varepsilon\le\varepsilon_0$, then
		$$
		\phi(\rho)\le c\left[\left(\frac{\rho}{R}\right)^\beta\phi(R)+B\rho^\beta\right],
		$$
		for all $0\le\rho\le R\le R_0$.
	\end{lemma}
	The next lemma is from \cite[Corollary 3.10]{MZ97}.
	\begin{lemma}\label{l2.2}
		If $u$ is a $p$-superharmonic function in $B_r$, i.e., $-\Delta_pu\ge0$, then for $\theta\in(0,1)$ and any $0<q\le p$, one has
		$$
		\sup_{B_{\theta r}}u^-\le\frac{C}{(1-\theta)^{n/q}}\frac{1}{|B_r|}\|u^-\|_{L^q(B_r)},
		$$
		where the constant $C>0$ depends only on $n$ and $p$.
	\end{lemma}
	We close the appendix by recalling the weak Harnack inequality from \cite[Theorem 3.13]{MZ97}.
	\begin{theorem}\label{t2.3}
		If $u$ is a $p$-superharmonic function in $\Omega$, and $0\le u\le M<\infty$ in some ball $B_r\subset\Omega$, for a constant $M$, then for any $\rho$, $\theta\in(0,1)$ and $s\in\left(0,\frac{n(p-1)}{n-p}\right)$, there exists a constant $C>0$, depending only on $p$, $n$, $\rho$, $\theta$ and $s$, such that 
		$$
		\frac{1}{|B_{\rho r}|}\|u\|_{L^s(B_{\rho r})}\le C\inf_{B_{\theta r}}u.
		$$
		In case $p=n$, the conclusion holds for any $s>0$.
	\end{theorem}

\end{document}